\documentclass[pdflatex,sn-mathphys-num
]{sn-jnl}
\usepackage[
]{hyperref}
\hypersetup{hypertexnames=false}
\usepackage{graphicx}%
\usepackage{multirow}%
\usepackage{amsmath,amssymb,amsfonts}%
\usepackage{amsthm}%
\usepackage{mathrsfs}%
\usepackage[title]{appendix}%
\usepackage{xcolor}%
\usepackage{textcomp}%
\usepackage{manyfoot}%
\usepackage{booktabs}%
\usepackage{listings}%

\usepackage{xspace,url,color,
underscore,import,diagbox}
\usepackage[shortlabels]{enumitem}
\usepackage{anyfontsize}
\theoremstyle{thmstyleone}%
\newtheorem{theorem}{Theorem}[section]

\newtheorem{lemma}[theorem]{Lemma}
\newtheorem{proposition}[theorem]{Proposition}

\theoremstyle{thmstyletwo}%
\newtheorem{example}{Example}[section]
\newtheorem{remark}{Remark}[section]

\theoremstyle{thmstylethree}%
\newtheorem{definition}{Definition}[section]

\DeclareUnicodeCharacter{00A0}{ }
\DeclareMathOperator*{\argmin}{arg\,min}
\DeclareMathOperator{\Ima}{Im}
\DeclareMathOperator{\dom}{dom}
\DeclareMathOperator{\ri}{ri}
\DeclareMathOperator{\inte}{int}
\DeclareMathOperator{\aff}{aff}
\DeclareMathOperator{\conv}{conv}
\DeclareMathOperator{\cl}{cl}
\DeclareMathOperator{\lip}{lip}
\DeclareMathOperator{\diag}{diag}

\DeclareMathOperator{\cone}{cone}
\DeclareMathOperator{\spn}{span}
\DeclareMathOperator{\gph}{gph}

\begin{document}
\raggedbottom

\title[VU-calculus and the $\mathcal{U}$-Hessian]
{A $\mathcal{VU}$-calculus for composite functions and the $\mathcal{U}$-Hessian of partly smooth functions}


\author*[1]{\fnm{Shuai} \sur{Liu}}\email{shuai0liu@gmail.com}

%

\affil*[1]{\orgdiv{School of Mathematics and Statistics}, \orgname{Nanfang College, Guangzhou}, \orgaddress{\street{No. 882, Wenquan Avenue, Conghua District}, \city{Guangzhou}, \postcode{510970}, \state{Guangdong}, \country{China}}}

%


\abstract{We assemble a finite-dimensional \(\mathcal{VU}\)-calculus for composite
nonsmooth functions whose outer function is convex: the chain rule, separable and
convex sums, a strictly differentiable perturbation of a lower semicontinuous (lsc) term, the
model \(\delta_X+f_0+\theta\circ F\), and a finite maximum of \(C^1\)
functions. The same algebra yields an \(\varepsilon\)-\(\mathcal{VU}\)
chain rule for a proper outer approximation of the subdifferential.
On the set \(\mathcal{R}_{h,F}\) of points at which the convex chain rule
holds, the subspaces \(\mathcal{V}f\) and \(\mathcal{U}f\), an orthonormal
frame of \(\mathcal{U}f\), and the \(\mathcal{U}\)-gradient \(\bar g_u\)
are written in terms of the factors.
If \(f\) is \(C^1\)-partly smooth at \(\bar{x}\), that gradient is
\(\bar g_u=U_f^\top\nabla_{\mathcal{M}}f(\bar{x})\).
If \(f\) is \(C^2\)-partly smooth and \(0\in\ri\partial f(\bar{x})\), the
convex \(\mathcal{U}\)-Hessian \(H_U\) (when \(f\) is convex) or the local
matrix \(H_\varepsilon\) (when \(f\) is prox-regular at \(\bar{x}\) for
\(0\)) equals the Gram matrix
\(U_f^\top\nabla^2_{\mathcal{M}}f(\bar{x})\,U_f\); the same matrix,
written along the active manifold in a continuous frame, depends
continuously on the base point.
If in addition \(f\) is \(C^2\)-partly smooth at \(\bar{x}\), then under
prox-regularity and subdifferential continuity at \(\bar{x}\) for \(0\),
\(H_\varepsilon\succ 0\) is equivalent to tilt stability of \(\bar{x}\)
and to strong metric regularity of \(\partial f\) at \((\bar{x},0)\), with
\(\lip\bigl((\partial f)^{-1}\bigr)(0\mid\bar{x})=\|H_\varepsilon^{-1}\|\).
The calculus and the test are illustrated on a hinge composite and two
elementary tests of \(H_U\succ 0\).}



\maketitle
\raggedbottom
\makeatletter
\def\@textbottom{\vskip \z@ \@plus 1pt}
\makeatother

\section{Introduction}

Nonsmooth objectives that arise as a convex outer function composed with a
smooth inner mapping, as a sum of such terms, or as a finite maximum of
\(C^1\) pieces, are typically nondifferentiable, yet the
nondifferentiability is confined to a proper subspace. The complementary
subspace consists of the directions along which the function looks
\(C^1\), and often \(C^2\). This splitting is the
\(\mathcal{VU}\)-decomposition of Lemar\'echal--Oustry--Sagastiz\'abal
\cite{Lemarechal2000}, later extended to proper lsc functions in
\cite{Liu2020}. Lewis \cite{Lewis2002active} identified the same geometry
with an active manifold of partial smoothness: by normal sharpness, the \(\mathcal{U}\)-space is
the tangent space and the \(\mathcal{V}\)-space is the normal space.
The relative-interior condition \(0\in\ri\partial f(\bar{x})\) enters only
later, for sharpness of the \(\mathcal{V}\)-slice and for the Gram
identification of \(H_U\) or \(H_\varepsilon\) with the covariant Hessian.
Mifflin--Sagastiz\'abal \cite{Mifflin2003pdg} record that identification
for primal-dual gradient structure and relate the \(U\)-Hessian to
second-order epi-derivatives; identifiability of the manifold at
\(\bar{x}\) for \(0\) is Hare--Lewis \cite{Hare2004}.

The subdifferential identity behind the chain rule is that of
Rockafellar and Wets~\cite[Theorem~10.6]{Rockafellar1998}; the
\(\mathcal{VU}\)-space formulae are the image and preimage of that
identity under \(\spn(\partial f-g)\). Separable and convex
sums likewise follow from the corresponding subdifferential rules.
A neighbouring composite class---a positively homogeneous convex outer
function composed with a smooth mapping that sends the reference point to
the origin---is studied in Shapiro \cite{Shapiro2003}.
A chain rule for \(\mathcal{VU}\)-spaces of a composite
\(f=h\circ\Phi\) is already available in
Hare--Planiden--Sagastiz\'abal \cite{Hare2020chain}, under primal-dual
gradient structure and a fast track for \(h\), a \(C^2\) inner mapping,
and transversality of \(\Phi\) to the active manifold of \(h\).
The present formulae use only convexity of the outer function, \(F\) of
class \(C^1\), and the two alternatives in \(\mathcal{R}_{h,F}\); the
\(\mathcal{U}\)-gradient is the coordinate vector
\(\bar g_u=U_f^\top g\) in an orthonormal frame of \(\mathcal{U}f\).
On the second-order side, the \(\mathcal{U}\)-Hessian
\(H_U=\nabla^2\mathcal{L}_U(0)\) (convex case) and the local matrix
\(H_\varepsilon=\nabla^2\mathcal{L}_\varepsilon(0)\) are the reduced
Hessians. For convex partly smooth functions with
\(0\in\ri\partial f(\bar{x})\) one has the Gram identity
\(H_U=U_f^\top\nabla^2_{\mathcal{M}}f(\bar{x})\,U_f\);
a related identification in the bivariate setting of
\cite{Liu2025} is recorded there. The present paper records the same
Gram identity for \(H_\varepsilon\) under prox-regularity.
Positive-definiteness of the covariant Hessian is
equivalent to tilt stability by
Poliquin--Rockafellar \cite{Poliquin1998tilt} and
Lewis--Zhang \cite{lewis2013partial}; the equivalences of that test with
\(H_\varepsilon\succ 0\) and with strong metric regularity of
\(\partial f\), together with the modulus
\(\lip\bigl((\partial f)^{-1}\bigr)(0\mid\bar{x})=\|H_\varepsilon^{-1}\|\),
are recorded in Section~\ref{ss:UH-tilt}.
Quadratic growth of \(\mathcal{L}_U\) at the origin, and its transfer to
\(f\) when \(f\) is convex, is Lemar\'echal--Oustry
\cite{Lemarechal2001Growth}.
Without partial smoothness, tilt stability is the uniform variant of
that growth \cite{Drusvyatskiy2013tilt}; under
\(C^2\)-partial smoothness the two conditions coincide
\cite[Theorem~6.3]{lewis2013partial}.

A \(\mathcal{VU}\)-step needs a frame of \(\mathcal{U}f\) and a
compressed curvature of size \(\dim\mathcal{U}f\). For the composite
models common in structured optimization those objects should be read
from the factors, rather than assembled as a generalized Hessian in
\(\mathbb{R}^n\). Existing composite \(\mathcal{VU}\) formulae either
require primal-dual gradient structure and a fast track, or stop at the
spaces without the frame, the \(\mathcal{U}\)-gradient, or the sign of
the reduced Hessian. The second-order tests used for tilt stability,
strong metric regularity of \(\partial f\), and the superlinear-rate
hypothesis of proximal-gradient \(\mathcal{VU}\) methods are stated in
several languages (covariant Hessian, \(\mathcal{U}\)-Hessian, partial
\(\mathcal{U}\)-Hessian, restricted Hessian on a fast track). Writing
them as one Gram matrix makes the algorithmic hypothesis
\(H_\varepsilon\succ 0\) the same test as tilt stability, and gives the
Lipschitz modulus of \((\partial f)^{-1}\) as \(\|H_\varepsilon^{-1}\|\).
That interface, not a new constraint qualification, is the reason for
the calculus below.

The contribution is therefore one of organisation and of interfaces, not of
new constraint qualifications or of a new characterisation of tilt
stability. Precisely:
\begin{itemize}
\item
a \(\mathcal{VU}\)-calculus for six constructions (chain rule, separable
and convex sums, a strictly differentiable perturbation, the model
\(\delta_X+f_0+\theta\circ F\), and a finite maximum), under convexity of
the outer function, together with the \(\varepsilon\)-\(\mathcal{VU}\)
push-forward of a proper outer approximation of \(\partial h\);
the set \(\mathcal{R}_{h,F}\) only names the two alternatives that give
the convex chain rule, and the finite-maximum formulae are classical;
\item
Gram identification: if \(f\) is \(C^2\)-partly smooth and
\(0\in\ri\partial f(\bar{x})\), then \(H_U\) (convex case) or
\(H_\varepsilon\) (prox-regular case) equals
\(U_f^\top\nabla^2_{\mathcal{M}}f(\bar{x})\,U_f\);
\item
that Gram matrix, recentred along \(\mathcal{M}\), depends continuously
on the base point in a continuous orthonormal frame of
\(T_{\mathcal{M}}\);
\item
under \(C^2\)-partial smoothness, prox-regularity and subdifferential
continuity at \(\bar{x}\) for \(0\), and \(0\in\ri\partial f(\bar{x})\),
\(H_\varepsilon\succ 0\) is equivalent to tilt stability of \(\bar{x}\)
and to strong metric regularity of \(\partial f\) at \((\bar{x},0)\), with
\(\lip\bigl((\partial f)^{-1}\bigr)(0\mid\bar{x})=\|H_\varepsilon^{-1}\|\);
\item
a nonconvex outer function, and an inner mapping that is merely Lipschitz,
are left aside.
\end{itemize}
The single-variable \(\mathcal{U}\)-Lagrangian of Definition~\ref{def:U-Lag} is that
of \cite{Lemarechal2000}, written in the coordinates \(U_fu+V_fw\);
we do not use the intrinsic splitting \(\oplus\).
The \(\varepsilon\)-\(\mathcal{VU}\) identity of
Subsection~\ref{ss:eps-VU} does not treat continuity of the subspaces in
\((x,\varepsilon)\), nor the choice of a small enlargement; those questions
remain those of \cite{Liu2019Subdifferential}.
Implementable \(\mathcal{VU}\)-methods begin with the convex algorithm of
Mifflin--Sagastiz\'abal \cite{Mifflin2005}; later schemes appear in
\cite{Liu2020VUMethods,Liu2025}.
Superlinear convergence and the Dennis--Mor\'e condition are taken up in
Section~\ref{sec:conclude}.
The calculus itself is first-order in the factors and second-order only
at a minimizer with \(0\in\ri\partial f\).

Section~\ref{sec:prelim} fixes the \(\mathcal{VU}\)-decomposition, frames,
the single-variable \(\mathcal{U}\)-Lagrangian and the \(\mathcal{U}\)-Hessian.
Section~\ref{sec:calculus} contains the exact calculus, the first half of a
soft-margin example on frozen features, and the \(\varepsilon\)-\(\mathcal{VU}\)
push-forward at the end of the section.
Section~\ref{sec:ps} records partial smoothness and derivatives along
the manifold (Subsection~\ref{ss:PS}), identifies \(H_U\) and
\(H_\varepsilon\) at \(\bar{x}\) (Subsection~\ref{ss:UH-ident}),
records the same matrix and its continuity along \(\mathcal{M}\)
(Subsection~\ref{ss:UH-along}), and records the
equivalences with tilt stability and quadratic growth
(Subsection~\ref{ss:UH-tilt}); three examples isolate
positive-definiteness and the relative-interior hypothesis.

\section{Preliminaries}
\label{sec:prelim}
This section fixes the \(\mathcal{VU}\)-spaces, frames, the
\(\mathcal{U}\)-gradient, the convex \(\mathcal{U}\)-Lagrangian
\(\mathcal{L}_U\), and the local \(\mathcal{U}\)-Lagrangian
\(\mathcal{L}_\varepsilon\) used in Sections~\ref{sec:calculus} and~\ref{sec:ps}.
We write \(\overline{\mathbb{R}}=\mathbb{R}\cup\{\pm\infty\}\) and
\(\Gamma_0(\mathbb{R}^n)\) for the proper convex lsc functions
\(\mathbb{R}^n\to\overline{\mathbb{R}}\).
The Euclidean unit ball is \(\mathbb{B}\); \(e_i\) is the \(i\)th
canonical vector; \(P_S\) is orthogonal projection onto \(S\).
The indicator and support function of \(S\) are \(\delta_S\) and
\(\sigma_S\). For a set we use the standard hulls
\(\conv\), \(\aff\), \(\ri\),
\(\inte\) and \(\cl\); for a function,
\(\dom f\). For a linear map, \(\Ima A\) and
\(\ker A\). The normal and tangent cones to \(S\) at \(x\in S\) are
\(N_S(x)\) and \(T_S(x)\).
The directional derivative of \(f\) at \(x\) in the direction \(d\) is
\(f'(x;d)\); the limiting subdifferential at \(\bar{x}\) is
\(\partial f(\bar{x})\).
If \(F:\mathbb{R}^n\to\mathbb{R}^m\) is of class \(C^1\), then
\(F'(x)\in\mathbb{R}^{m\times n}\) is the Jacobian; its transpose
\(F'(x)^\top\) is the adjoint \(\mathbb{R}^m\to\mathbb{R}^n\).
For a set \(S\subset\mathbb{R}^m\) the symbol
\(F'(x)^{-1}(S)\) means the preimage
\(\{y\in\mathbb{R}^n:F'(x)y\in S\}\), not a matrix inverse.
The same prime denotes the derivative of a scalar \(C^1\) map; we write
\(\nabla f_i(x)\) when the gradient vector is needed.

The \(\mathcal{VU}\) decomposition was initially defined in \cite{Lemarechal2000} for a finite-valued convex function and later extended to proper lsc functions in \cite[Definition 5]{Liu2020}. 
\begin{definition}[\(\mathcal{VU}\) decomposition]\label{def:vu}
Given a proper, lsc function \(f\) and a point \(\bar{x}\) with \(\partial f(\bar{x})\neq\emptyset\), the \(\mathcal{VU}\) decomposition associated with \(\partial f(\bar{x})\) is defined by
\begin{equation*}
\mathcal{V}f(\bar{x})\mathrel{\mathop:}= \spn(\partial f(\bar{x})-{g} ),\quad \mathcal{U}f(\bar{x})\mathrel{\mathop:}=\mathcal{V}f(\bar{x})^\perp
\end{equation*}
where \({g}\) is an arbitrary subgradient in \(\partial f(\bar{x})\).
\end{definition}
The relative interior of the subdifferential \(\partial f(\bar{x})\) is written
\[
\ri\partial f(\bar{x})
\;:=\;
\bigl\{
g\in\partial f(\bar{x})
:
g+\inte(\eta\mathbb{B})\cap\mathcal{V}f(\bar{x})
\subset\partial f(\bar{x})
\text{ for some }\eta>0
\bigr\}.
\]
\begin{remark}[equivalence with the Rockafellar relative interior]
\label{rem:ri-Rock}
For any \(g\in\partial f(\bar{x})\) one has
\(\mathcal{V}f(\bar{x})=\aff\partial f(\bar{x})-g\). If in addition \(\partial f(\bar{x})\)
is convex, the display above is precisely the Rockafellar relative interior
of the convex set \(\partial f(\bar{x})\) taken in the affine hull
\(\aff\partial f(\bar{x})\):
\[
\ri\partial f(\bar{x})
=\bigl\{g\in\partial f(\bar{x}):
g+\inte\bigl(\eta\mathbb{B}\bigr)\cap\bigl(\aff\partial f(\bar{x})-g\bigr)
\subset\partial f(\bar{x})
\text{ for some }\eta>0\bigr\}.
\]
The set \(\partial f(\bar{x})\) is convex at every point of subdifferential
regularity, and therefore at every point of partial smoothness
(Definition~\ref{def:PS}(ii) below). In that setting the working definition
and the Rockafellar \(\ri\) coincide, so the hypothesis
\(0\in\ri\partial f(\bar{x})\) used after Lewis \cite{Lewis2002active},
Poliquin--Rockafellar \cite{Poliquin1998tilt} and Lewis--Zhang
\cite{lewis2013partial} is the standard one.
\end{remark}
If \(f\) is subdifferentially regular at \(\bar{x}\), then according to \cite[Proposition 5]{Liu2020} (see also \cite[Lemma 12]{Hare2001quadratic}) the subspaces in question can also be characterized as
\begin{gather*}
\mathcal{U}f(\bar{x})=N_{\partial f(\bar{x})}(g^\circ)=\bigl\{w\in\mathbb{R}^n:f'(\bar{x};-w)=-f'(\bar{x};w)\bigr\},\\
\mathcal{V}f(\bar{x})=T_{\partial f(\bar{x})}(g^\circ).
\end{gather*}
\begin{definition}[orthonormal frame of the $\mathcal{U}$-space]\label{def:frame}
Let \( f \) admit a \(\mathcal{VU}\)-decomposition at \( \bar{x} \).
A matrix \( U_f(\bar{x})\in\mathbb{R}^{n\times\dim\mathcal{U}f(\bar{x})} \) is called an
\emph{orthonormal frame} of \( \mathcal{U}f(\bar{x}) \) if its columns form an orthonormal basis of \( \mathcal{U}f(\bar{x}) \).
We write \( U_f \) when the point is clear.
The associated orthogonal projector is
\[
P_{\mathcal{U}f(\bar{x})}=U_fU_f^\top.
\]
Frames of \( \mathcal{U}h(F(\bar{x})) \), \( \mathcal{U}\theta(F(\bar{x})) \), \( \mathcal{U}f_i(\bar{x}) \)
are denoted \( U_h \), \( U_\theta \), \( U_{f_i} \) respectively.
\end{definition}
Any two frames differ by a right orthogonal factor \( Q \); the projector \( P_{\mathcal{U}f(\bar{x})}=U_fU_f^\top \) is independent of the choice.

Let \(m=\dim\mathcal{U}f(\bar{x})\) and \(n-m=\dim\mathcal{V}f(\bar{x})\). Given an orthonormal frame \(U_f\) of \(\mathcal{U}f(\bar{x})\) and a basis matrix \(V_f\in\mathbb{R}^{n\times(n-m)}\) of \(\mathcal{V}f(\bar{x})\), every \(z\in\mathbb{R}^n\) decomposes uniquely as
\[
z=U_fz_u+V_fz_v,
\qquad
z_u=U_f^\top z\in\mathbb{R}^m,
\qquad
z_v=V_f^\dagger z\in\mathbb{R}^{n-m},
\]
where \(V_f^\dagger:=(V_f^\top V_f)^{-1}V_f^\top\) is the Moore--Penrose inverse of \(V_f\).

\begin{definition}[$\mathcal{U}$-gradient]
\label{def:U-grad}
The set \(U_f^\top\partial f(\bar{x})\) is a singleton. Following
\cite{Liu2025}, we write
\begin{equation}
\label{eq:U-grad}
\bar g_u
\;:=\;
U_f^\top g
\in\mathbb{R}^{\dim\mathcal{U}f(\bar{x})}
\qquad\text{for every }g\in\partial f(\bar{x})
\end{equation}
for its unique element, and call \(\bar g_u\) the \(\mathcal{U}\)-gradient of
\(f\) at \(\bar{x}\) in the frame \(U_f\). The
\emph{embedded \(U\)-gradient} is the vector in \(\mathbb{R}^n\)
\[
g_U(\bar{x})
\;:=\;
U_f\bar g_u
=
P_{\mathcal{U}f(\bar{x})}g
\in
\mathcal{U}f(\bar{x}).
\]
It does not depend on the choice of \(g\in\partial f(\bar{x})\), nor on
the frame beyond the product \(U_f\bar g_u\). This is the vector
written \(g_U\) in \cite{Miller2005}. When the point varies we write
\(\bar g_u(x)\) and \(g_U(x)\). When several functions appear we write
\(\bar g_u(f)\), \(\bar g_u(h)\), \(\bar g_u(\theta)\), and so on.
\end{definition}

\begin{definition}[$\mathcal{U}$-Lagrangian and $\mathcal{U}$-Hessian]
\label{def:U-Lag}
Let \(f\in\Gamma_0(\mathbb{R}^n)\) and let \(\bar{x}\) satisfy
\(\partial f(\bar{x})\neq\emptyset\).
Following \cite{Lemarechal2000}, fix \(\bar{g}\in\partial f(\bar{x})\).
The \emph{\(U\)-Lagrangian} of \(f\) at \(\bar{x}\) associated with \(\bar{g}\) is
\begin{equation*}
\mathcal{L}_U^{\bar{g}}(u)
\;:=\;
\inf_{w\in\mathbb{R}^{n-m}}
\Bigl(
f(\bar{x}+U_fu+V_fw)-\bigl\langle \bar{g}_v,\,V_f^\top V_fw\bigr\rangle
\Bigr),
\qquad
u\in\mathbb{R}^m.
\end{equation*}
The associated set of global \(\mathcal{V}\)-space minimizers is
\[
W(u;\bar{g}_v)
\;:=\;
\argmin_{w\in\mathbb{R}^{n-m}}
\Bigl(
f(\bar{x}+U_fu+V_fw)-\bigl\langle \bar{g}_v,\,V_f^\top V_fw\bigr\rangle
\Bigr).
\]
When \(0\in\partial f(\bar{x})\) we take \(\bar{g}=0\) once and for all, and write
\begin{equation}
\label{eq:U-Lag-0}
\mathcal{L}_U(u)
\;:=\;
\inf_{w\in\mathbb{R}^{n-m}}
f(\bar{x}+U_fu+V_fw),
\qquad
u\in\mathbb{R}^m.
\end{equation}
This is the convention used throughout the paper whenever
\(f\) is convex and \(0\in\partial f(\bar{x})\).
Whenever the second derivative exists, the \emph{\(\mathcal{U}\)-Hessian} of \(f\) at \(\bar{x}\) is
\begin{equation}
\label{eq:U-Hess}
H_U
\;:=\;
\nabla^2\mathcal{L}_U(0)\in\mathbb{R}^{m\times m}.
\end{equation}
Existence of \(H_U\) is not assumed here.
\end{definition}

\begin{definition}[local \(\mathcal{U}\)-Lagrangian]
\label{def:U-Lag-loc}
Let \(f:\mathbb{R}^n\to\overline{\mathbb{R}}\) be proper and lsc, and let
\(\bar{x}\) satisfy \(\partial f(\bar{x})\neq\emptyset\).
Following \cite[Definition~6]{Liu2020}, fix \(\varepsilon>0\) and an
arbitrary \(\bar{g}\in\ri\partial f(\bar{x})\). The
\emph{local \(U\)-Lagrangian} of \(f\) at \(\bar{x}\) associated with
\(\bar{g}\) is
\begin{equation}
\label{eq:L-eps}
\mathcal{L}_\varepsilon^{\bar{g}}(u)
\;:=\;
\inf_{\|w\|\le\varepsilon}
\Bigl(
f(\bar{x}+U_fu+V_fw)
-
\bigl\langle\bar{g}_v,\,V_f^\top V_fw\bigr\rangle
\Bigr),
\qquad
u\in\mathbb{R}^m.
\end{equation}
The associated set of \(\mathcal{V}\)-space minimizers is
\begin{equation}
\label{eq:W-eps}
W_\varepsilon(u;\bar{g}_v)
\;:=\;
\argmin_{\|w\|\le\varepsilon}
\Bigl(
f(\bar{x}+U_fu+V_fw)
-
\bigl\langle\bar{g}_v,\,V_f^\top V_fw\bigr\rangle
\Bigr).
\end{equation}
When \(0\in\ri\partial f(\bar{x})\) we take \(\bar{g}=0\) and write
\(\mathcal{L}_\varepsilon\) and \(W_\varepsilon(u)\) for
\(\mathcal{L}_\varepsilon^{0}\) and \(W_\varepsilon(u;0)\).
Whenever the second derivative exists, the
\emph{local \(\mathcal{U}\)-Hessian} is
\begin{equation}
\label{eq:U-Hess-eps}
H_\varepsilon
\;:=\;
\nabla^2\mathcal{L}_\varepsilon(0)\in\mathbb{R}^{m\times m}.
\end{equation}
\end{definition}
\begin{remark}[when \(\mathcal{L}_U=\mathcal{L}_\varepsilon\)]
\label{rem:LU-Leps}
Definition~\ref{def:U-Lag} applies only to convex functions.
For a general proper lsc function the working objects are
\(\mathcal{L}_\varepsilon\) and \(H_\varepsilon\) from
Definition~\ref{def:U-Lag-loc}.
If \(f\in\Gamma_0(\mathbb{R}^n)\), then
\(\mathcal{L}_U\le\mathcal{L}_\varepsilon\) and the two coincide near
the origin: the set of global \(\mathcal{V}\)-minimizers satisfies
\(W(u;\bar{g}_v)=o(\|u\|)\), hence lies in the ball
\(\|w\|\le\varepsilon\) for small \(u\)
\cite{Lemarechal2000,Liu2020}. In that case
\(H_U=H_\varepsilon\) whenever either second derivative exists.
Without a quadratic minorant the value of \(\mathcal{L}_\varepsilon^{\bar{g}}\)
may still depend on \(\varepsilon\)
\cite[Remark~3]{Liu2020}.
\end{remark}

\section{A \texorpdfstring{\(\mathcal{VU}\)}{VU}-calculus}\label{sec:calculus}
This section records the \(\mathcal{VU}\)-spaces of standard composite models.
Throughout, the outer function in a composition is convex and the inner mapping
is of class \(C^1\). Nonconvex outer functions and merely Lipschitz inner
mappings are left aside (see Section~\ref{sec:conclude}).
The identities themselves are those of the subspaces \(\mathcal{V}f\) and
\(\mathcal{U}f\). An orthonormal frame and the \(\mathcal{U}\)-gradient
are recorded immediately after each rule.
Subsection~\ref{ss:chain} treats composition;
Subsection~\ref{ss:sums} treats sums and the structured model;
Subsection~\ref{ss:frames} computes a hinge example;
Subsection~\ref{ss:eps-VU} repeats the same algebra for an outer
approximation of the subdifferential.

The \(\mathcal{VU}\)-chain rule uses the subdifferential chain rule of
Rockafellar--Wets \cite[Theorem~10.6]{Rockafellar1998}: if
\(F:\mathbb{R}^n\to\mathbb{R}^m\) is of class \(C^1\) and
\(h:\mathbb{R}^m\to\overline{\mathbb{R}}\) is proper and lsc, then
\begin{equation*}
\label{eq:CR}
\partial(h\circ F)(x)
\;=\;
F'(x)^\top\partial h\bigl(F(x)\bigr)
\tag{CR}
\end{equation*}
at every \(x\in\dom(h\circ F)\) at which \(h\) is
subdifferentially regular at \(F(x)\) and
\begin{equation}
\label{eq:RW-CQ}
\ker F'(x)^\top \cap \partial^\infty h\bigl(F(x)\bigr)
\;=\;
\{0\}.
\end{equation}

\begin{remark}[orthonormal frames of kernels]
\label{lem:frame-ker}
Let \(A\in\mathbb{R}^{p\times n}\) and set \(\mathcal{L}:=\ker A\). If
\(A=P\Sigma Q^\top\) is a singular-value decomposition of \(A\), the columns of
\(Q\) associated with vanishing singular values form an orthonormal frame of
\(\mathcal{L}\). Any two frames of \(\mathcal{L}\) differ by a right orthogonal
factor, and the orthogonal projector onto \(\mathcal{L}\) is independent of the
choice.
\end{remark}

\begin{definition}[admissible points]
\label{def:R-hF}
Let \(F:\mathbb{R}^n\to\mathbb{R}^m\) be of class \(C^1\) and let
\(h\in\Gamma_0(\mathbb{R}^m)\). The set \(\mathcal{R}_{h,F}\) is contained in \(\dom(h\circ F)\) by construction: for \(\bar{x}\in\dom(h\circ F)\) we write
\(\bar{x}\in\mathcal{R}_{h,F}\) if either
\begin{itemize}
\item[(R1)]
\(F(\bar{x})\in\ri(\dom h)\) and \(F\) maps a neighbourhood of \(\bar{x}\) into
\(\aff(\dom h)\), or
\item[(R2)]
\(F(\bar{x})\in\dom h\setminus\ri(\dom h)\)
and \eqref{eq:RW-CQ} holds.
\end{itemize}
\end{definition}
When \(h\) is convex it is subdifferentially regular on its domain, so
equality in \eqref{eq:CR} follows from \eqref{eq:RW-CQ} by
\cite[Theorem~10.6]{Rockafellar1998}, and from (R1) by
Lemma~\ref{lem:R1-CR} below.

\begin{lemma}[relative-interior chain rule]
\label{lem:R1-CR}
Let \(h\in\Gamma_0(\mathbb{R}^m)\) and let \(F:\mathbb{R}^n\to\mathbb{R}^m\) be
of class \(C^1\). If \(\bar{x}\in\dom(h\circ F)\),
\(F(\bar{x})\in\ri(\dom h)\), and \(F\) maps a
neighbourhood of \(\bar{x}\) into \(\aff(\dom h)\),
then \(\partial(h\circ F)(\bar{x})=F'(\bar{x})^\top\partial h(F(\bar{x}))\).
\end{lemma}
\begin{proof}
Set \(\hat h(z):=h\bigl(F(\bar{x})+z\bigr)\) and
\(\hat F(x):=F(x)-F(\bar{x})\), and write
\(L:=\aff(\dom\hat h)=\aff(\dom h)-F(\bar{x})\).
Then \(L\) is a linear subspace, \(0\in\ri(\dom\hat h)\), and
\(\hat F\) maps a neighbourhood of \(\bar{x}\) into \(L\).
Let \(\hat h^L\) be the restriction of \(\hat h\) to \(L\), viewed as a
function on \(L\). As functions on \(\mathbb{R}^m\) one has
\(\hat h=\hat h^L+\delta_L\).
Since \(0\in\operatorname{int}_L(\dom\hat h^L)\), the function
\(\hat h^L\) is locally Lipschitz on \(L\) near the origin, and
\(\partial^\infty\hat h^L(0)=\{0\}\).
As \(\hat F\) takes values in \(L\) near \(\bar{x}\),
\cite[Theorem~10.6]{Rockafellar1998} yields
\(\partial(h\circ F)(\bar{x})=\partial(\hat h^L\circ\hat F)(\bar{x})
=\hat F'(\bar{x})^\top\partial\hat h^L(0)\).
The sum rule for \(\hat h=\hat h^L+\delta_L\) gives
\(\partial\hat h(0)=\partial\hat h^L(0)+L^\perp\).
Translation of the subdifferential and
\(\hat F'(\bar{x})=F'(\bar{x})\) then give
\(\partial h(F(\bar{x}))=\partial\hat h(0)=\partial\hat h^L(0)+L^\perp\).
Finally \(\operatorname{range}\hat F'(\bar{x})\subset L\), so
\(\hat F'(\bar{x})^\top L^\perp=\{0\}\) and
\(\partial(h\circ F)(\bar{x})=F'(\bar{x})^\top\partial h(F(\bar{x}))\).
\end{proof}

\begin{remark}[the set \(\mathcal{R}_{h,F}\) is not a new qualification]
\label{rem:R-hF}
Condition (R2) is the qualification of
\cite[Theorem~10.6]{Rockafellar1998}. Condition (R1) is the
relative-interior requirement together with the neighbourhood
condition of Lemma~\ref{lem:R1-CR}. If \(\aff(\dom h)\neq\mathbb{R}^m\),
then \(\partial^\infty h=(\aff(\dom h))^\perp\) throughout
the relative interior, so (R1) need not imply (R2).
\end{remark}

\begin{example}[(R1) need not imply (R2)]
\label{ex:R1-not-R2-CR}
Let
\[
h(y_1,y_2)
=
\begin{cases}
|y_1|&\text{if }y_2=0,\\
+\infty&\text{if }y_2\neq 0,
\end{cases}
\qquad
F(x)=(x,0).
\]
Then \(h\in\Gamma_0(\mathbb{R}^2)\), \(F\) is of class \(C^\infty\),
\(F(x)\in\ri(\dom h)\) for every \(x\), and the image of \(F\) lies
in \(\aff(\dom h)\), so (R1) holds.
Here \(\partial^\infty h(F(x))=\{0\}\times\mathbb{R}\) and
\(\ker F'(x)^\top=\spn\{(0,1)\}\), so \eqref{eq:RW-CQ}
fails and (R2) fails. Nevertheless \(h\circ F=|\cdot|\) and
\eqref{eq:CR} holds by Lemma~\ref{lem:R1-CR}.
\end{example}

\subsection{Chain rule}
\label{ss:chain}
The set \(\mathcal{R}_{h,F}\) is now in hand. The exact
\(\mathcal{VU}\)-chain rule under that qualification is as follows.
\begin{theorem}[Chain rule for convex outer function]\label{thm:chain}
Let \(f:\mathbb{R}^n\to\overline{\mathbb{R}}\), \(F:\mathbb{R}^n\to\mathbb{R}^m\) and \(h:\mathbb{R}^m\to\overline{\mathbb{R}}\) satisfy \(f=h\circ F\), with \(h\in\Gamma_0(\mathbb{R}^m)\) and \(F\) of class \(C^1\).
If \(\bar{x}\in \mathcal{R}_{h,F}\) and \(\partial h\bigl(F(\bar{x})\bigr)\neq\emptyset\)
then 
\begin{enumerate}[label=\textnormal{(\roman*)}]
\item \(\mathcal{V}f(\bar{x})=F'(\bar{x})^\top\mathcal{V}h\bigl(F(\bar{x})\bigr)\) and
\(\mathcal{U}f(\bar{x})=F'(\bar{x})^{-1}\bigl(\mathcal{U}h(F(\bar{x}))\bigr)\).
\label{it0:1}
\item If \(F'(\bar{x})\) has full row rank \(m\), then
\(\mathcal{U}h\bigl(F(\bar{x})\bigr)=F'(\bar{x})\,\mathcal{U}f(\bar{x})\). 
\end{enumerate}
\end{theorem}
\begin{proof}
By \cite[Theorem~10.6]{Rockafellar1998} in case~(R2), or by
Lemma~\ref{lem:R1-CR} in case~(R1),
the chain rule \eqref{eq:CR} is available at \(\bar{x}\in\mathcal{R}_{h,F}\).
Together with \(\partial h(F(\bar{x}))\neq\emptyset\) this yields
\(\partial f(\bar{x})=F'(\bar{x})^\top\partial h(F(\bar{x}))\neq\emptyset\), so
the spaces \(\mathcal{V}f(\bar{x})\) and \(\mathcal{U}f(\bar{x})\) are defined.
Let \(s^h\in\partial h(F(\bar{x}))\) be arbitrary and set \(s^f:=F'(\bar{x})^\top s^h\). Then
\begin{align*}
\mathcal{V}f(\bar{x})
&=\spn\bigl(\partial f(\bar{x})-s^f\bigr)
=\spn\bigl(F'(\bar{x})^\top\bigl(\partial h(F(\bar{x}))-s^h\bigr)\bigr)\\
&=F'(\bar{x})^\top\spn\bigl(\partial h(F(\bar{x}))-s^h\bigr)
=F'(\bar{x})^\top\mathcal{V}h\bigl(F(\bar{x})\bigr),
\end{align*}
where the third equality follows from the elementary identity \(\spn(AC)=A(\spn C)\).  
Taking orthogonal complements and using the standard identity \((MV)^\perp=\{y:M^\top y\in V^\perp\}\) for a linear map \(M\) and a subspace \(V\) gives
\[
\mathcal{U}f(\bar{x})=\bigl\{y:F'(\bar{x})y\in\mathcal{U}h\bigl(F(\bar{x})\bigr)\bigr\}.
\]

For (ii), write \( A:=F'(\bar{x}) \) and \( U:=\mathcal{U}h(F(\bar{x})) \). Part (i) says \( \mathcal{U}f(\bar{x})=A^{-1}(U) \). For any linear map \( A \) and any subspace \( U \) of the codomain, one has \(
A\bigl(A^{-1}(U)\bigr)=U\cap\Ima A
\). 
Full row rank means that \( A \) is surjective, hence \( \Ima A=\mathbb{R}^m \) and
\(
A\,\mathcal{U}f(\bar{x})=U=\mathcal{U}h\bigl(F(\bar{x})\bigr),
\) 
which is the claim.
\end{proof}
Let the columns of \(V_h\) span \(\mathcal{V}h\bigl(F(\bar{x})\bigr)\). Then
\(\mathcal{U}f(\bar{x})=\ker A\) for
\(A=V_h^\top F'(\bar{x})\)
(equivalently \(A=(I-U_hU_h^\top)F'(\bar{x})\)), and an orthonormal frame
\(U_f\) of \(\mathcal{U}f(\bar{x})\) is furnished by Remark~\ref{lem:frame-ker}.
For every \(s\in\partial h\bigl(F(\bar{x})\bigr)\),
\(\bar g_u(f)=U_f^\top F'(\bar{x})^\top s\).

The hypothesis \(\partial h(F(\bar{x}))\neq\emptyset\) in Theorem~\ref{thm:chain}
is used only so that the spaces \(\mathcal{V}f(\bar{x})\) and
\(\mathcal{U}f(\bar{x})\) are defined; it is not required for \eqref{eq:CR}.
Without a point of \(\mathcal{R}_{h,F}\) the space identity can fail.

\begin{example}[failure of the chain rule without a qualification]
\label{ex:chain-CQ}
Let \(C\subset\mathbb{R}^2\) be the closed unit disk centred at \((0,1)\), so
that \(C\) meets the horizontal axis only at the origin. Set
\(F:\mathbb{R}\to\mathbb{R}^2\), \(F(t)=(t,0)\), and \(h=\delta_C\). Then
\(h\circ F=\delta_{\{0\}}\), \(F'(0)=(1,0)^\top\), and
\[
\ker F'(0)^\top=\spn\{(0,1)\},\qquad
\partial^\infty h(0,0)=N_C(0,0)=\cone\{(0,-1)\}.
\]
The intersection is nontrivial, so \(0\notin\mathcal{R}_{h,F}\). Here
\(\mathcal{V}(h\circ F)(0)=\mathbb{R}\), whereas
\(F'(0)^\top\mathcal{V}h(0,0)=\{0\}\).
\end{example}

\begin{remark}[space identities versus the subdifferential chain rule]
\label{rem:V-vs-CR}
Once \eqref{eq:CR} holds, the identity
\(\spn(AC)=A(\spn C)\) gives the
\(\mathcal{VU}\)-chain rule of Theorem~\ref{thm:chain}(i). The converse
need not hold: the space identities
\(\mathcal{V}(h\circ F)=F'^\top\mathcal{V}h\) and
\(\mathcal{U}(h\circ F)=F'^{-1}(\mathcal{U}h)\) may be valid while
\eqref{eq:CR} fails. A mere inclusion
\(F'^\top\partial h\subset\partial(h\circ F)\) is not enough to guarantee
either conclusion. Theorem~\ref{thm:chain} therefore uses \eqref{eq:CR},
not the equality of spans. The next example records a case in which the
\(\mathcal{VU}\)-formulae hold and \eqref{eq:CR} does not;
Example~\ref{ex:chain-CQ} is the complementary case in which both fail.
\end{remark}

\begin{example}[relative interior without the neighbourhood condition]
\label{ex:R1-not-R2}
Let \(h\) be as in Example~\ref{ex:R1-not-R2-CR} and set \(F(x)=(x,x^2)\).
Then \(F(0)\in\ri(\dom h)\), but no neighbourhood of the origin is
mapped into the affine hull, so \(0\notin\mathcal{R}_{h,F}\).
Here \(h\circ F=\delta_{\{0\}}\), hence
\(\partial(h\circ F)(0)=\mathbb{R}\) and
\(\mathcal{V}(h\circ F)(0)=\mathbb{R}\),
\(\mathcal{U}(h\circ F)(0)=\{0\}\).
On the other hand
\(F'(0)^\top=(1,0)\) and \(\partial h(0,0)=[-1,1]\times\mathbb{R}\), so
\(F'(0)^\top\partial h(0,0)=[-1,1]\neq\mathbb{R}\) and \eqref{eq:CR} fails,
while
\[
F'(0)^\top\mathcal{V}h(0,0)=\mathbb{R}
=\mathcal{V}(h\circ F)(0),\qquad
F'(0)^{-1}\bigl(\mathcal{U}h(0,0)\bigr)=\{0\}
=\mathcal{U}(h\circ F)(0).
\]
The \(\mathcal{VU}\)-chain rule holds; the subdifferential chain rule does
not.
\end{example}

A finite maximum of \(C^1\) functions has the following
\(\mathcal{VU}\)-spaces. The same two formulae appear as the
\(\varepsilon=0\) case of
\cite[§5.4]{Liu2019Subdifferential}, and as the finitely determined
max-function in Mifflin--Sagastiz\'abal \cite[§4.1]{Mifflin2000pdg}
(there written
\(\mathcal{V}f=\spn\{\nabla f_i-\nabla f_{i_0}:i\in I\setminus\{i_0\}\}\)).
Neither source uses affine independence of the active gradients for the
spaces themselves; that extra assumption is only for a full-rank PDG
representation. The statement is recorded only as a corollary of
Theorem~\ref{thm:chain}.
\begin{proposition}[finite max of smooth functions, classical]\label{prop:max}
Let \( f=\max\{f_1,\dots,f_m\} \) with each \( f_i:\mathbb{R}^n\to\mathbb{R} \) of class \( C^1 \), and let
\[
I(\bar{x}):=\{i:f_i(\bar{x})=f(\bar{x})\}.
\]
Then \( f \) admits a \(\mathcal{VU}\)-decomposition at \( \bar{x} \) and
\begin{align*}
\mathcal{V}f(\bar{x})
&=\spn\bigl\{\nabla f_i(\bar{x})-\nabla f_j(\bar{x}):i,j\in I(\bar{x})\bigr\},\\
\mathcal{U}f(\bar{x})
&=\bigl\{y\in\mathbb{R}^n:\langle\nabla f_i(\bar{x}),y\rangle=\langle\nabla f_j(\bar{x}),y\rangle\text{ for all }i,j\in I(\bar{x})\bigr\}.
\end{align*}
\end{proposition}
\begin{proof}
Write \( f=h\circ F \) with
\begin{gather}
F:\mathbb{R}^n\to\mathbb{R}^m,\qquad F(x)=\bigl(f_1(x),\dots,f_m(x)\bigr), \text{ and }
\\
h:\mathbb{R}^m\to\mathbb{R},\qquad h(y)=\max\{y_1,\dots,y_m\}.
\end{gather}
The outer function \(h\) is finite and convex on \(\mathbb{R}^m\), so
\(\dom h=\mathbb{R}^m\), the neighbourhood condition in (R1) is
automatic, and \(\partial^\infty h\equiv\{0\}\). Every point, including
\(\bar{x}\), belongs to \(\mathcal{R}_{h,F}\). Also
\[
\partial h(F(\bar{x}))=\conv\{e_i:i\in I(\bar{x})\}\neq\emptyset,
\]
so both functions admit a \(\mathcal{VU}\)-decomposition and Theorem~\ref{thm:chain} applies.

Let \( I:=I(\bar{x}) \). The subspace parallel to \( \aff(\partial h(F(\bar{x}))) \) is
\[
\mathcal{V}h\bigl(F(\bar{x})\bigr)
=\spn\{e_i-e_j:i,j\in I\}
=\bigl\{v\in\mathbb{R}^m:v_k=0\text{ for }k\notin I,\ \textstyle\sum_{i\in I}v_i=0\bigr\}.
\]
Its orthogonal complement is \(
\mathcal{U}h\bigl(F(\bar{x})\bigr)
=\bigl\{w\in\mathbb{R}^m:w_i=w_j\text{ for all }i,j\in I\bigr\}.
\)
The Jacobian of \( F \) satisfies \( F'(\bar{x})^\top v=\sum_{i=1}^m v_i\nabla f_i(\bar{x}) \). Theorem \ref{thm:chain}(i) therefore gives
\begin{align*}
\mathcal{V}f(\bar{x})
&=F'(\bar{x})^\top\mathcal{V}h\bigl(F(\bar{x})\bigr)
=\Bigl\{\sum_{i\in I}v_i\nabla f_i(\bar{x}):\sum_{i\in I}v_i=0\Bigr\}\\
&=\spn\bigl\{\nabla f_i(\bar{x})-\nabla f_j(\bar{x}):i,j\in I\bigr\} \text{ and }
\end{align*}
\begin{align*}
\mathcal{U}f(\bar{x})
&=\bigl\{y:F'(\bar{x})y\in\mathcal{U}h(F(\bar{x}))\bigr\}
=\bigl\{y:\langle\nabla f_i(\bar{x}),y\rangle\text{ are equal for all }i\in I\bigr\}.
\end{align*}
\end{proof}
With \(V_h\) the matrix of columns \(e_i-e_{i_*}\) for \(i\in I\setminus\{i_*\}\),
the rows of \(A=V_h^\top F'(\bar{x})\) are
\(\bigl(\nabla f_i(\bar{x})-\nabla f_{i_*}(\bar{x})\bigr)^\top\).
Then \(\mathcal{U}f(\bar{x})=\ker A\), a frame \(U_f\) is given by
Remark~\ref{lem:frame-ker}, and
\(\bar g_u(f)=U_f^\top\nabla f_i(\bar{x})\) for every active index \(i\).

\subsection{Sums and the structured model}
\label{ss:sums}

The chain rule treats a single outer function. The same linear algebra
gives the spaces of a separable sum, a convex sum, a smooth perturbation
and the structured model \(\delta_X+f_0+\theta\circ F\).

\begin{lemma}[separable functions]
\label{lem:separable}
Let \( f(x_1,x_2)=f_1(x_1)+f_2(x_2) \), where \( f_i:\mathbb{R}^{n_i}\to\overline{\mathbb{R}} \) (\( i=1,2 \)) are proper lsc functions.  
Assume that at \( \bar{x}=(\bar{x}_1,\bar{x}_2) \) one has 
each \( f_i \) is subdifferentially regular at \( \bar{x}_i \), and \( \partial f_i(\bar{x}_i)\neq\emptyset \) (\( i=1,2 \)).  
Then 
\[
\mathcal{V}f(\bar{x})=\mathcal{V}f_1(\bar{x}_1)\times\mathcal{V}f_2(\bar{x}_2),\qquad
\mathcal{U}f(\bar{x})=\mathcal{U}f_1(\bar{x}_1)\times\mathcal{U}f_2(\bar{x}_2).
\]
This is the \(\mathcal{VU}\)-form of \cite[Proposition~10.5]{Rockafellar1998}.
\end{lemma}
\begin{proof}
Subdifferential regularity of each \( f_i \) together with \( \partial f_i(\bar{x}_i)\neq\emptyset \) and \cite[Theorem~8.30]{Rockafellar1998} yields \( df_i(\bar{x}_i)(0)=\sigma_{\partial f_i(\bar{x}_i)}(0)=0 \).  
Thus the hypotheses of \cite[Proposition~10.5]{Rockafellar1998} hold, and we obtain
\[
\partial f(\bar{x})=\partial f_1(\bar{x}_1)\times\partial f_2(\bar{x}_2)\neq\emptyset,
\]
so \( f \) admits a \(\mathcal{VU}\)-decomposition.

Fix \( g=(g_1,g_2)\in\partial f(\bar{x}) \). By definition and the elementary identity \( \spn(A\times B)=\spn A\times\spn B \),
\begin{align*}
\mathcal{V}f(\bar{x})
&=\spn\bigl(\partial f(\bar{x})-g\bigr)
=\spn\bigl((\partial f_1(\bar{x}_1)-g_1)\times(\partial f_2(\bar{x}_2)-g_2)\bigr)\\
&=\spn(\partial f_1(\bar{x}_1)-g_1)\times\spn(\partial f_2(\bar{x}_2)-g_2)
=\mathcal{V}f_1(\bar{x}_1)\times\mathcal{V}f_2(\bar{x}_2).
\end{align*}
Taking orthogonal complements and using \( (V_1\times V_2)^\perp=V_1^\perp\times V_2^\perp \) gives the claimed formula for \( \mathcal{U}f(\bar{x}) \).
\end{proof}
If \(U_{f_i}\) frames \(\mathcal{U}f_i(\bar{x}_i)\), then
\(U_f=\operatorname{diag}(U_{f_1},U_{f_2})\) and
\(\bar g_u(f)=U_f^\top g\) for every \(g=(g_1,g_2)\) with \(g_i\in\partial f_i(\bar{x}_i)\).

\begin{proposition}[sum rule for convex functions]\label{prop:convex-sum}
Let \( f_1,f_2:\mathbb{R}^n\to\overline{\mathbb{R}} \) be proper, lsc and convex, and let \( x\in\mathbb{R}^n \) satisfy \( \partial f_1(x)\neq\emptyset \) and \( \partial f_2(x)\neq\emptyset \).
Assume either
\begin{enumerate}
\item[(R1')] \( x\in\ri(\dom f_1)\cap\ri(\dom f_2) \), or
\item[(R2')] \( \partial^\infty f_1(x)\cap\bigl(-\partial^\infty f_2(x)\bigr)=\{0\} \).
\end{enumerate}
Then \( f:=f_1+f_2 \) admits a \(\mathcal{VU}\)-decomposition at \( x \) and
\[
\mathcal{V}f(x)=\mathcal{V}f_1(x)+\mathcal{V}f_2(x),\qquad
\mathcal{U}f(x)=\mathcal{U}f_1(x)\cap\mathcal{U}f_2(x).
\]
In particular, the conclusion holds when \(f_1\) or \(f_2\) is real-valued or locally Lipschitz continuous near \(x\). 
\end{proposition}
\begin{proof}
Write \( f=g\circ F \) with \( F(z)=(z,z) \) and \( g(y_1,y_2)=f_1(y_1)+f_2(y_2) \).
Then \( g \) is proper, lsc and convex,
\[
\dom g=\dom f_1\times\dom f_2,\qquad
\partial g(x,x)=\partial f_1(x)\times\partial f_2(x)\neq\emptyset,
\]
and, since \( g \) is a separable convex sum,
\[
\partial^\infty g(x,x)=\partial^\infty f_1(x)\times\partial^\infty f_2(x).
\]
Convex functions are subdifferentially regular wherever the subdifferential is nonempty, so Lemma~\ref{lem:separable} gives
\(
\mathcal{U}g(x,x)=\mathcal{U}f_1(x)\times\mathcal{U}f_2(x).
\)
Here \(\ri(\dom g)=\ri(\dom f_1)\times\ri(\dom f_2)\), so
\(F(x)\in\ri(\dom g)\) if and only if (R1') holds.
The neighbourhood clause of (R1) is read in
\(\aff(\dom f_1)\cap\aff(\dom f_2)\): the diagonal map sends that
hull into \(\aff(\dom g)\), hence \(x\in\mathcal{R}_{g,F}\) under (R1').
On the other hand,
\[
F'(x)=\begin{pmatrix}I\\ I\end{pmatrix},\qquad
\ker F'(x)^\top=\bigl\{(u,-u):u\in\mathbb{R}^n\bigr\},
\]
hence
\begin{align*}
\ker F'(x)^\top\cap\partial^\infty g(F(x))
&=\bigl\{(u,-u):u\in\partial^\infty f_1(x),\ -u\in\partial^\infty f_2(x)\bigr\}\\
&=\partial^\infty f_1(x)\cap\bigl(-\partial^\infty f_2(x)\bigr).
\end{align*}
Thus \( \ker F'(x)^\top\cap\partial^\infty g(F(x))=\{0\} \) if and only if (R2') holds. Theorem  \ref{thm:chain} therefore applies and yields
\(
\mathcal{U}f(x)=\bigl\{d:F'(x)d\in\mathcal{U}g(F(x))\bigr\}.
\)
Since \( F'(x)d=(d,d) \) and \( \mathcal{U}g(F(x))=\mathcal{U}f_1(x)\times\mathcal{U}f_2(x) \),
\[
\mathcal{U}f(x)=\mathcal{U}f_1(x)\cap\mathcal{U}f_2(x).
\]
Taking orthogonal complements and using \( (U_1\cap U_2)^\perp=U_1^\perp+U_2^\perp \) for subspaces \(U_1,U_2\) we obtain
\[
\mathcal{V}f(x)=\mathcal{V}f_1(x)+\mathcal{V}f_2(x).
\] 

The special cases follow from the property that the horizon subdifferential becomes null under those circumstances. Specifically, if \(f_i\) is real-valued then as an lsc convex function it has \(\partial^\infty f_i(x)=N_{\dom f_i}(x)=\{0\}\). If \(f_i\) is locally Lipschitz continuous near \(x\) then \(\partial^\infty f_i(x)=\{0\}\).  
\end{proof}
If \(U_{f_i}\) frames \(\mathcal{U}f_i(x)\), set
\(A=(I-U_{f_2}U_{f_2}^\top)U_{f_1}\). Then
\(\mathcal{U}f(x)=U_{f_1}(\ker A)\) and \(U_f=U_{f_1}Q_0\), where the columns
of \(Q_0\) frame \(\ker A\). For every \(g_i\in\partial f_i(x)\),
\(\bar g_u(f)=U_f^\top(g_1+g_2)\).

For special nonconvex functions we have the following result. 
\begin{proposition}[strictly differentiable perturbation]
\label{prop:smooth-lsc}
Let \( f_1,f_2:\mathbb{R}^n\to\overline{\mathbb{R}} \) and \( f:=f_1+f_2 \). 
Assume that at a point \( x\in\mathbb{R}^n \)
the function \( f_1 \) is strictly differentiable, \( f_2 \) is locally lsc and \( \partial f_2(x)\neq\emptyset \).
Then
\[
\mathcal{V}f(x)=\mathcal{V}f_2(x),\qquad
\mathcal{U}f(x)=\mathcal{U}f_2(x).
\]
\end{proposition}
\begin{proof}
By \cite[Exercise 10.10]{Rockafellar1998} strict differentiability of \( f_1 \) at \( x \) yields
\[
\partial f(x)=\nabla f_1(x)+\partial f_2(x)\,.
\]
Since \(\partial f_2(x)\neq\emptyset\), \( f \) admits a \(\mathcal{VU}\)-decomposition at \( x \). Fix \( g_2\in\partial f_2(x) \) and set \( g:=\nabla f_1(x)+g_2\in\partial f(x) \). By definition,
\begin{align*}
\mathcal{V}f(x)
&=\spn\bigl(\partial f(x)-g\bigr)
=\spn\bigl(\nabla f_1(x)+\partial f_2(x)-\nabla f_1(x)-g_2\bigr)\\
&=\spn\bigl(\partial f_2(x)-g_2\bigr)
=\mathcal{V}f_2(x).
\end{align*}
Taking orthogonal complements yields \( \mathcal{U}f(x)=\mathcal{U}f_2(x) \).
\end{proof}
Thus \(U_f=U_{f_2}\) and
\(\bar g_u(f)=U_f^\top\bigl(\nabla f_1(x)+g_2\bigr)\) for every
\(g_2\in\partial f_2(x)\).

\begin{proposition}[structured model \(\delta_X+f_0+\theta\circ F\)]
\label{prop:model}
Let $X\subset\mathbb{R}^n$ be nonempty, closed and convex, let
$\theta\in\Gamma_0(\mathbb{R}^m)$, and let $f_0:\mathbb{R}^n\to\mathbb{R}$
and $F:\mathbb{R}^n\to\mathbb{R}^m$ be of class $C^1$.
Set $f:=\delta_X+f_0+\theta\circ F$ and assume $\bar{x}\in X$ satisfies $\partial\theta(F(\bar{x}))\neq\emptyset$.
Write \(H(x)=(x,F(x))\).
Assume either
\begin{enumerate}
\item[(M1)] \(\bar{x}\in\ri(X)\), \(F(\bar{x})\in\ri(\dom\theta)\), and
\(H\) maps a neighbourhood of \(\bar{x}\) into
\(\aff(X)\times\aff(\dom\theta)\), or
\item[(M2)]
\(\partial^\infty\theta(F(\bar{x}))
\cap
\bigl(F'(\bar{x})^\top\bigr)^{-1}\bigl(-N_X(\bar{x})\bigr)
=\{0\}\).
\end{enumerate}
Then $f$ admits a $\mathcal{VU}$-decomposition at $\bar{x}$ and
\begin{align*}
\mathcal{V}f(\bar{x})
&=F'(\bar{x})^\top\mathcal{V}\theta\bigl(F(\bar{x})\bigr)+\spn N_X(\bar{x}),\\
\mathcal{U}f(\bar{x})
&=F'(\bar{x})^{-1}\bigl(\mathcal{U}\theta(F(\bar{x}))\bigr)\cap\bigl(\spn N_X(\bar{x})\bigr)^\perp.
\end{align*}
\end{proposition}

\begin{proof}
Since $f_0$ is $C^1$, it is strictly differentiable at $\bar{x}$. Proposition~\ref{prop:smooth-lsc} applied to $f=f_0+(\delta_X+\theta\circ F)$ gives
\[
\mathcal{V}f(\bar{x})=\mathcal{V}(\delta_X+\theta\circ F)(\bar{x}),
\qquad
\mathcal{U}f(\bar{x})=\mathcal{U}(\delta_X+\theta\circ F)(\bar{x}),
\]
once the $\mathcal{VU}$-decomposition of $\delta_X+\theta\circ F$ is known to exist.

Write $\delta_X+\theta\circ F=\psi\circ H$ with $\psi(u,v)=\delta_X(u)+\theta(v)$.
Then $\psi$ is proper, lsc and convex, and
$\partial\psi(H(\bar{x}))=N_X(\bar{x})\times\partial\theta(F(\bar{x}))\neq\emptyset$.
Convex functions are subdifferentially regular wherever the subdifferential is nonempty, so Lemma~\ref{lem:separable} yields
\[
\mathcal{V}\psi(H(\bar{x}))=\spn N_X(\bar{x})\times\mathcal{V}\theta(F(\bar{x})),
\qquad
\mathcal{U}\psi(H(\bar{x}))=\bigl(\spn N_X(\bar{x})\bigr)^\perp\times\mathcal{U}\theta(F(\bar{x})),
\]
where we have used $\mathcal{V}\delta_X(\bar{x})=\spn N_X(\bar{x})$.

The assumption (M1) or (M2) is precisely \(\bar{x}\in\mathcal{R}_{\psi,H}\).
Indeed, \(\dom\psi=X\times\dom\theta\), hence
\(\ri(\dom\psi)=\ri(X)\times\ri(\dom\theta)\) and
\(\aff(\dom\psi)=\aff(X)\times\aff(\dom\theta)\).
Thus (M1) is exactly (R1) for \(\psi\circ H\).
On the other hand,
$H'(\bar{x})=\begin{pmatrix}I\\ F'(\bar{x})\end{pmatrix}$ and
$\ker H'(\bar{x})^\top=\{(-F'(\bar{x})^\top b,\,b):b\in\mathbb{R}^m\}$,
while $\partial^\infty\psi(H(\bar{x}))=N_X(\bar{x})\times\partial^\infty\theta(F(\bar{x}))$.
A pair $(-F'(\bar{x})^\top b,\,b)$ lies in this product if and only if
$b\in\partial^\infty\theta(F(\bar{x}))$ and $F'(\bar{x})^\top b\in -N_X(\bar{x})$.
Therefore $\ker H'(\bar{x})^\top\cap\partial^\infty\psi(H(\bar{x}))=\{0\}$ if and only if (M2) holds.

Theorem~\ref{thm:chain} applied to $\psi\circ H$ now gives
\begin{align*}
\mathcal{V}(\delta_X+\theta\circ F)(\bar{x})
&=H'(\bar{x})^\top\mathcal{V}\psi(H(\bar{x}))
=\spn N_X(\bar{x})+F'(\bar{x})^\top\mathcal{V}\theta(F(\bar{x})),\\
\mathcal{U}(\delta_X+\theta\circ F)(\bar{x})
&=\{y:H'(\bar{x})y\in\mathcal{U}\psi(H(\bar{x}))\}
=\bigl(\spn N_X(\bar{x})\bigr)^\perp
\cap
F'(\bar{x})^{-1}\bigl(\mathcal{U}\theta(F(\bar{x}))\bigr).
\end{align*}
Combining this with the reduction by $f_0$ completes the proof.
\end{proof}
Let the columns of \(N_{\bar{x}}\) span \(\spn N_X(\bar{x})\) and those of
\(V_\theta\) span \(\mathcal{V}\theta\bigl(F(\bar{x})\bigr)\). Then
\(\mathcal{U}f(\bar{x})=\ker A\) for
\[
A=\begin{pmatrix} N_{\bar{x}}^\top \\ V_\theta^\top F'(\bar{x})\end{pmatrix},
\]
a frame \(U_f\) is given by Remark~\ref{lem:frame-ker}, and for every
\(s\in\partial\theta\bigl(F(\bar{x})\bigr)\),
\(\bar g_u(f)=U_f^\top\bigl(\nabla f_0(\bar{x})+F'(\bar{x})^\top s\bigr)\).

\subsection{A soft-margin example}
\label{ss:frames}

The next example records the \(\mathcal{VU}\)-spaces, an orthonormal frame
and the \(\mathcal{U}\)-gradient of a convex classification head on frozen
features, via Proposition~\ref{prop:model}. The \(\mathcal{U}\)-Hessian is
taken up at the end of Section~\ref{sec:ps}.

\begin{example}[soft-margin classifier on frozen features]
\label{ex:probe}
Let \(\phi:\mathcal{X}\to\mathbb{R}^n\) be a fixed feature map (for instance
the output of a trained network, held constant). Given labelled samples
\(\{(x_i,y_i)\}_{i=1}^N\subset\mathcal{X}\times\{\pm1\}\) and \(C>0\), set
\[
f(w,b)
=
\frac12\|w\|^2
+
C\sum_{i=1}^N
\max\bigl\{0,\,1-y_i(\langle w,\phi(x_i)\rangle+b)\bigr\},
\qquad
(w,b)\in\mathbb{R}^n\times\mathbb{R}.
\]
This is the unconstrained hinge form of the primal soft-margin objective
on frozen features \(\phi(x_i)\) \cite[(24)]{Cortes1995}.
It is the structured model
\(f=\delta_X+f_0+\theta\circ F\) with \(X=\mathbb{R}^{n+1}\),
\(f_0(w,b)=\frac12\|w\|^2\),
\[
F:\mathbb{R}^{n+1}\to\mathbb{R}^N,
\qquad
F(w,b)
=
\bigl(
1-y_i(\langle w,\phi(x_i)\rangle+b)
\bigr)_{i=1}^N,
\]
and \(\theta(z)=C\sum_{i=1}^N\max\{0,z_i\}\).
Here \(\delta_X\equiv0\), \(f_0\) is \(C^2\), \(F\) is affine and
\(\theta\in\Gamma_0(\mathbb{R}^N)\) is finite everywhere, so
\(\aff(\dom\theta)=\mathbb{R}^N\) and every point belongs to
\(\mathcal{R}_{\theta,F}\).

\noindent\textbf{\(\mathcal{VU}\)-spaces.}
At a point \((\bar w,\bar b)\) write
\[
I_0
:=
\bigl\{i:F_i(\bar w,\bar b)=0\bigr\},
\qquad
I_>
:=
\bigl\{i:F_i(\bar w,\bar b)>0\bigr\},
\qquad
I_<
:=
\bigl\{i:F_i(\bar w,\bar b)<0\bigr\}
\]
for the active, violated and inactive margin sets. Write \(\sigma(t)=\max\{0,t\}\), so
that \(\theta(z)=C\sum_{i=1}^N\sigma(z_i)\). The convex subdifferential of
\(\sigma\) is \(\{0\}\) for \(t<0\), \(\{1\}\) for \(t>0\), and \([0,1]\) at
the origin. Lemma~\ref{lem:separable} therefore yields
\[
\partial\theta\bigl(F(\bar w,\bar b)\bigr)
=
C\prod_{i=1}^N\partial\sigma\bigl(F_i(\bar w,\bar b)\bigr),
\]
which is a product of singletons except in the coordinates \(i\in I_0\),
where the factor is the interval \([0,C]\). The subspace parallel to the
affine hull of this product is
\[
\mathcal{V}\theta\bigl(F(\bar w,\bar b)\bigr)
=
\spn\{e_i:i\in I_0\},
\]
and its orthogonal is
\(\mathcal{U}\theta(F(\bar w,\bar b))=\{z\in\mathbb{R}^N:z_i=0\text{ for all }i\in I_0\}\).
Active hinges therefore contribute independent directions \(e_i\), not
merely gradient differences as in Proposition~\ref{prop:max}.

The mapping \(F\) is affine, with \(i\)th row
\(\bigl(-y_i\phi(x_i)^\top,\,-y_i\bigr)\). Proposition~\ref{prop:model} gives
\(\mathcal{V}f(\bar w,\bar b)=F'(\bar{w},\bar{b})^\top\mathcal{V}\theta(F(\bar w,\bar b))\).
The generators are
\[
F'(\bar{w},\bar{b})^\top e_i
=
\bigl(-y_i\phi(x_i),\,-y_i\bigr),
\qquad
i\in I_0,
\]
and changing sign does not alter the span, so
\[
\mathcal{V}f(\bar w,\bar b)
=
\spn
\bigl\{
\bigl(y_i\phi(x_i),\,y_i\bigr):i\in I_0
\bigr\}.
\]
The \(\mathcal{U}\)-space is the preimage
\(\mathcal{U}f(\bar w,\bar b)=F'(\bar{w},\bar{b})^{-1}\bigl(\mathcal{U}\theta(F(\bar w,\bar b))\bigr)\).
A pair \((dw,db)\) lies in this preimage if and only if every active
coordinate of \(F'(\bar{w},\bar{b})(dw,db)\) vanishes, that is
\[
-y_i\bigl(\langle dw,\phi(x_i)\rangle+db\bigr)=0
\quad\text{for all }i\in I_0.
\]
Hence
\[
\mathcal{U}f(\bar w,\bar b)
=
\bigl\{
(dw,db)
:
y_i\bigl(\langle dw,\phi(x_i)\rangle+db\bigr)=0
\text{ for all }i\in I_0
\bigr\}.
\]

\noindent\textbf{Frames and the \(\mathcal{U}\)-gradient.}
\label{ex:probe-frame}
The kernel matrix recorded after Proposition~\ref{prop:model} has one row
\(y_i\bigl(\phi(x_i)^\top,\,1\bigr)\) for each \(i\in I_0\); then
\(\mathcal{U}f(\bar w,\bar b)=\ker A\) and an orthonormal frame \(U_f\) is
furnished by Remark~\ref{lem:frame-ker}.

A subgradient of \(f\) is
\[
(\bar w,0)
+
\sum_{i\in I_>}C\bigl(-y_i\phi(x_i),-y_i\bigr)
+
\sum_{i\in I_0}\lambda_i\bigl(-y_i\phi(x_i),-y_i\bigr),
\qquad
\lambda_i\in[0,C].
\]
The last sum lies in \(\mathcal{V}f(\bar w,\bar b)\).
The formula for \(\bar g_u(f)\) after Proposition~\ref{prop:model} therefore gives
\[
\bar g_u(f)
=
U_f^\top
\Biggl(
(\bar w,0)
+
\sum_{i\in I_>}C\bigl(-y_i\phi(x_i),-y_i\bigr)
\Biggr).
\]
When \(I_>=\emptyset\) the second summand vanishes. Partitioning
\(U_f=\begin{pmatrix}U_w\\ u_b^\top\end{pmatrix}\) with
\(U_w\in\mathbb{R}^{n\times n_U}\) and \(n_U=\dim\mathcal{U}f(\bar w,\bar b)\) then yields
\[
\bar g_u(f)
=
U_w^\top\bar w,
\qquad
\bigl\langle\bar g_u(f),U_f^\top(dw,db)\bigr\rangle
=
\langle\bar w,dw\rangle
\quad\text{for all }(dw,db)\in\mathcal{U}f(\bar w,\bar b).
\]
(The pairing contains no \(db\) term because the \(b\)-component of
\(\nabla f_0\) is zero; the constraint on \(db\) is already encoded in
\(\mathcal{U}f\).)
\end{example}

\subsection{\texorpdfstring{$\varepsilon$}{epsilon}-\texorpdfstring{\(\mathcal{VU}\)}{VU}-spaces}
\label{ss:eps-VU}

The chain rule of Theorem~\ref{thm:chain} is exact. The same linear algebra applies to an outer approximation of the subdifferential.
The construction below does not depend on sublinearity of the outer function, nor on polyhedral enlargements; it only uses a convex-valued outer approximation of \(\partial f\).

\begin{definition}[proper approximation of the subdifferential]
\label{def:proper-approx}
Let \(f:\mathbb{R}^n\to\overline{\mathbb{R}}\) be proper and lsc.
A set-valued mapping \((x,\varepsilon)\mapsto\bar\partial_\varepsilon f(x)\) with convex values is a
\emph{proper approximation} of \(\partial f\) if
\[
\partial f(x)\subset\bar\partial_\varepsilon f(x)
\qquad\text{for all }x\text{ and all }\varepsilon\ge 0,
\qquad
\bar\partial_0 f(x)=\partial f(x).
\]
\end{definition}

Two standard instances are the convex \(\varepsilon\)-subdifferential when \(f\) is finite and convex, and the set of \(\varepsilon\)-regular subgradients of \cite[Proposition~10.46]{Rockafellar1998}.
We do not impose the sandwich
\(\partial f(x)\subseteq\bar\partial_\varepsilon f(x)\subseteq\partial_\varepsilon f(x)\)
of \cite{Liu2019Subdifferential}, and we do not discuss continuity of the
spaces in \((x,\varepsilon)\). Those questions, and the choice of a small
structure-specific enlargement, remain those of that paper.

\begin{definition}[\(\varepsilon\)-\(\mathcal{VU}\) decomposition]
\label{def:epsVU}
Let \(f\) be proper and lsc, let \(\bar x\) satisfy \(\partial f(\bar x)\neq\emptyset\), and let \(\bar\partial_\varepsilon f\) be a proper approximation of \(\partial f\).
The associated \(\varepsilon\)-\(\mathcal{VU}\) spaces at \(x\) (with \(\bar\partial_\varepsilon f(x)\neq\emptyset\)) are
\begin{align}
\mathcal{V}_\varepsilon f(x)
&:=
\spn\bigl(\bar\partial_\varepsilon f(x)-g\bigr),
\qquad
g\in\bar\partial_\varepsilon f(x)\text{ arbitrary},
\label{def:vep}\\
\mathcal{U}_\varepsilon f(x)
&:=
N_{\bar\partial_\varepsilon f(x)}(g^\circ),
\qquad
g^\circ\in\ri\bar\partial_\varepsilon f(x)\text{ arbitrary}.
\label{def:uep}
\end{align}
\end{definition}

\begin{proposition}[equivalent descriptions of the \(\varepsilon\)-\(\mathcal{U}\)-space]
\label{prop:uepsaltdef}
In Definition~\ref{def:epsVU},
\begin{align}
\mathcal{U}_\varepsilon f(x)
&=
\bigl\{
w
:
\langle g-g^\circ,w\rangle=0
\text{ for all }g\in\bar\partial_\varepsilon f(x)
\bigr\}
\quad\text{for every }g^\circ\in\ri\bar\partial_\varepsilon f(x),
\label{eq:epsUfequi}\\
\mathcal{U}_\varepsilon f(x)
&=
\mathcal{V}_\varepsilon f(x)^\perp.
\label{eq:4}
\end{align}
In particular both spaces are independent of the auxiliary vectors \(g\) and \(g^\circ\), and \(\mathbb{R}^n=\mathcal{V}_\varepsilon f(x)\oplus\mathcal{U}_\varepsilon f(x)\).
\end{proposition}
\begin{proof}
The set \(\bar\partial_\varepsilon f(x)\) is convex by Definition~\ref{def:proper-approx}, so the normal cone in \eqref{def:uep} is
\[
\mathcal{U}_\varepsilon f(x)
=
\bigl\{
w
:
\langle g-g^\circ,w\rangle\le 0
\text{ for all }g\in\bar\partial_\varepsilon f(x)
\bigr\}.
\]
This contains the right-hand side of \eqref{eq:epsUfequi}. For the reverse inclusion, fix \(w\) in the normal cone and \(g\in\bar\partial_\varepsilon f(x)\). If \(g=g^\circ\) there is nothing to prove. Otherwise the relative-interior description of \(g^\circ\) yields \(\eta>0\) such that
\[
g^\circ+\bigl(\eta\mathbb{B}\cap\mathcal{V}_\varepsilon f(x)\bigr)
\subset
\bar\partial_\varepsilon f(x).
\]
The unit vector \(v:=-(g-g^\circ)/\|g-g^\circ\|\) lies in \(\mathcal{V}_\varepsilon f(x)\), hence \(g^\circ+\eta v\in\bar\partial_\varepsilon f(x)\). The normal-cone inequality at this point reads
\[
0
\ge
\langle w,\eta v\rangle
=
-\frac{\eta}{\|g-g^\circ\|}\langle g-g^\circ,w\rangle,
\]
so \(\langle g-g^\circ,w\rangle\ge 0\). Combined with the opposite inequality already available, one obtains \(\langle g-g^\circ,w\rangle=0\). Replacing \(g^\circ\) by any other point of the relative interior does not change the annihilator of \(\bar\partial_\varepsilon f(x)-g^\circ\), so \(\mathcal{U}_\varepsilon f(x)\) is independent of that choice.

If \(w\in\mathcal{V}_\varepsilon f(x)^\perp\), then \(\langle w,z-g^\circ\rangle=0\) for every \(z\in\aff\bar\partial_\varepsilon f(x)\), hence for every \(z\in\bar\partial_\varepsilon f(x)\), and \(w\) satisfies \eqref{eq:epsUfequi}. Conversely, \eqref{eq:epsUfequi} annihilates the span of \(\bar\partial_\varepsilon f(x)-g^\circ\), which is \(\mathcal{V}_\varepsilon f(x)\). Thus \(\mathcal{U}_\varepsilon f(x)=\mathcal{V}_\varepsilon f(x)^\perp\).
\end{proof}

When \(f=h\circ F\) and \(\bar x\in\mathcal{R}_{h,F}\), the exact chain rule \eqref{eq:CR} pushes a proper approximation of \(\partial h\) forward to a proper approximation of \(\partial f\):
\begin{equation}
\label{eq:bardfeps}
\bar\partial_\varepsilon f(x)
\;:=\;
F'(x)^\top\bar\partial_\varepsilon h\bigl(F(x)\bigr),
\qquad
x\in\dom(h\circ F).
\end{equation}
Indeed \(\partial f(x)=F'(x)^\top\partial h(F(x))\subset F'(x)^\top\bar\partial_\varepsilon h(F(x))\) on \(\mathcal{R}_{h,F}\), and the two sides coincide at \(\varepsilon=0\). The values remain convex.

\begin{proposition}[\(\varepsilon\)-\(\mathcal{VU}\) chain rule]
\label{prop:epsVUchain}
Let \(f=h\circ F\) with \(h\in\Gamma_0(\mathbb{R}^m)\) and \(F:\mathbb{R}^n\to\mathbb{R}^m\) of class \(C^1\).
Let \(\bar\partial_\varepsilon h\) be a proper approximation of \(\partial h\), and define \(\bar\partial_\varepsilon f\) by \eqref{eq:bardfeps}.
If \(x\in\mathcal{R}_{h,F}\) and \(\bar\partial_\varepsilon h(F(x))\neq\emptyset\), then
\begin{align}
\mathcal{V}_\varepsilon f(x)
&=
F'(x)^\top\mathcal{V}_\varepsilon h\bigl(F(x)\bigr),
\label{eq:uvefcom}\\
\mathcal{U}_\varepsilon f(x)
&=
\bigl\{
y
:
F'(x)y\in\mathcal{U}_\varepsilon h\bigl(F(x)\bigr)
\bigr\}.
\label{eq:Uepsfbyh}
\end{align}
\end{proposition}
\begin{proof}
The set \(\bar\partial_\varepsilon f(x)\) is nonempty and convex, so both \(\varepsilon\)-\(\mathcal{VU}\) decompositions are defined.
Pick \(s\in\bar\partial_\varepsilon h(F(x))\) and set \(g:=F'(x)^\top s\in\bar\partial_\varepsilon f(x)\). Then
\begin{align*}
\mathcal{V}_\varepsilon f(x)
&=
\spn\bigl(\bar\partial_\varepsilon f(x)-g\bigr)
=
\spn\bigl(F'(x)^\top\bigl(\bar\partial_\varepsilon h(F(x))-s\bigr)\bigr)\\
&=
F'(x)^\top\spn\bigl(\bar\partial_\varepsilon h(F(x))-s\bigr)
=
F'(x)^\top\mathcal{V}_\varepsilon h\bigl(F(x)\bigr),
\end{align*}
where the third equality is \(\spn(AC)=A(\spn C)\).
The identity for \(\mathcal{U}_\varepsilon f(x)\) is then the orthogonal complement, exactly as in Theorem~\ref{thm:chain}(i), using Proposition~\ref{prop:uepsaltdef}.
The only additional ingredient is that \eqref{eq:CR} remains available on \(\mathcal{R}_{h,F}\), so that \eqref{eq:bardfeps} is a proper approximation.
\end{proof}

The exact spaces of Theorem~\ref{thm:chain} are the case \(\varepsilon=0\). No rank assumption on \(F'(x)\) is used. If \(F'(x)\) has full row rank, the same argument as in Theorem~\ref{thm:chain}(ii) gives \(\mathcal{U}_\varepsilon h(F(x))=F'(x)\,\mathcal{U}_\varepsilon f(x)\).

\begin{remark}[\(\varepsilon\)-\(\mathcal{VU}\) forms of the preceding rules]
\label{rem:eps-sums-frames}
Assume \(\bar\partial_\varepsilon(\,\cdot\,)\) is a proper approximation
in the sense of Definition~\ref{def:proper-approx}.
The identities of this section persist after the substitutions below,
in the same order as the exact rules: finite max, separable sum, convex
sum, strictly differentiable perturbation, structured model, then frames
and the \(\mathcal{U}\)-gradient.
\begin{enumerate}[label=\textnormal{(\alph*)}]
\item
\emph{Finite max}
(Proposition~\ref{prop:max}).
Write \(I_\varepsilon:=\{i:f_i(\bar{x})\ge f(\bar{x})-\varepsilon\}\)
for the \(\varepsilon\)-active index set. For
\(\bar\partial_\varepsilon f=\conv\{\nabla f_i:i\in I_\varepsilon\}\)
one has
\(\mathcal{V}_\varepsilon f=\spn\{\nabla f_i-\nabla f_j:i,j\in I_\varepsilon\}\)
and
\(\mathcal{U}_\varepsilon f=\{y:\langle\nabla f_i(\bar{x}),y\rangle\text{ coincide for }i\in I_\varepsilon\}\);
cf.\ \cite[§5.4]{Liu2019Subdifferential}.
\item
\emph{Separable sum}
(Lemma~\ref{lem:separable}).
\(\bar\partial_\varepsilon f:=\bar\partial_\varepsilon f_1\times\bar\partial_\varepsilon f_2\)
gives
\(\mathcal{V}_\varepsilon f=\mathcal{V}_\varepsilon f_1\times\mathcal{V}_\varepsilon f_2\)
and
\(\mathcal{U}_\varepsilon f=\mathcal{U}_\varepsilon f_1\times\mathcal{U}_\varepsilon f_2\).
Regularity of each factor is required at \(\varepsilon=0\).
\item
\emph{Convex sum}
(Proposition~\ref{prop:convex-sum}).
Under (R1') or (R2'),
\(\bar\partial_\varepsilon f:=\bar\partial_\varepsilon f_1+\bar\partial_\varepsilon f_2\)
gives
\(\mathcal{V}_\varepsilon f=\mathcal{V}_\varepsilon f_1+\mathcal{V}_\varepsilon f_2\)
and
\(\mathcal{U}_\varepsilon f=\mathcal{U}_\varepsilon f_1\cap\mathcal{U}_\varepsilon f_2\).
\item
\emph{Strictly differentiable perturbation}
(Proposition~\ref{prop:smooth-lsc}).
\(\bar\partial_\varepsilon f:=\nabla f_1+\bar\partial_\varepsilon f_2\)
gives
\(\mathcal{V}_\varepsilon f=\mathcal{V}_\varepsilon f_2\)
and
\(\mathcal{U}_\varepsilon f=\mathcal{U}_\varepsilon f_2\).
\item
\emph{Structured model}
(Proposition~\ref{prop:model}).
With \(\bar\partial_\varepsilon\delta_X=N_X\) and
\(\bar\partial_\varepsilon(\theta\circ F)=F'^\top\bar\partial_\varepsilon\theta\),
the alternatives (M1) or (M2) give
\(\mathcal{V}_\varepsilon f=F'^\top\mathcal{V}_\varepsilon\theta+\spn N_X\)
and
\(\mathcal{U}_\varepsilon f=F'^{-1}(\mathcal{U}_\varepsilon\theta)\cap(\spn N_X)^\perp\).
Enlarging \(N_X\) to an \(\varepsilon\)-normal cone changes the second
summand.
\item
\emph{Frames.}
The kernel matrices recorded after each rule, with \(\varepsilon\)-frames
in place of exact frames, have kernel \(\mathcal{U}_\varepsilon f\).
\item
\emph{\(\mathcal{U}\)-gradient.}
\(\bar g_u^\varepsilon(f):=U_{\varepsilon,f}^\top g\)
is independent of \(g\in\bar\partial_\varepsilon f\), and equals the
exact \(\mathcal{U}\)-gradient projected onto \(\mathcal{U}_\varepsilon f\)
whenever \(\mathcal{U}_\varepsilon f\subset\mathcal{U}f\).
\end{enumerate}
Continuity in \((x,\varepsilon)\) is not claimed.
\end{remark}

\noindent\textit{Continuation of Example~\ref{ex:probe}.}
\label{ex:probe-eps}
Since \(X=\mathbb{R}^{n+1}\) one has \(N_X\equiv\{0\}\), so
Remark~\ref{rem:eps-sums-frames}(e) reduces to the \(\varepsilon\)-chain
rule for \(\theta\circ F\):
\[
\mathcal{V}_\varepsilon f(\bar w,\bar b)
=
F'(\bar w,\bar b)^\top\mathcal{V}_\varepsilon\theta\bigl(F(\bar w,\bar b)\bigr),
\qquad
\mathcal{U}_\varepsilon f(\bar w,\bar b)
=
F'(\bar w,\bar b)^{-1}\bigl(\mathcal{U}_\varepsilon\theta(F(\bar w,\bar b))\bigr).
\]
The exact subdifferential of \(\theta\) is a product of singletons except
in the coordinates \(i\in I_0\). A proper approximation that fattens only
those kinks is
\[
I_\varepsilon
:=
\bigl\{i:\bigl|F_i(\bar w,\bar b)\bigr|\le\varepsilon\bigr\},
\qquad
\bar\partial_\varepsilon\theta\bigl(F(\bar w,\bar b)\bigr)
:=
C\prod_{i\in I_\varepsilon}[0,1]
\times
\prod_{i\notin I_\varepsilon}
\partial\sigma\bigl(F_i(\bar w,\bar b)\bigr).
\]
The second product consists of singletons, so
\(\mathcal{V}_\varepsilon\theta(F(\bar w,\bar b))=\spn\{e_i:i\in I_\varepsilon\}\)
and
\[
\mathcal{V}_\varepsilon f(\bar w,\bar b)
=
\spn
\bigl\{
\bigl(y_i\phi(x_i),\,y_i\bigr):i\in I_\varepsilon
\bigr\},
\]
\[
\mathcal{U}_\varepsilon f(\bar w,\bar b)
=
\bigl\{
(dw,db)
:
y_i\bigl(\langle dw,\phi(x_i)\rangle+db\bigr)=0
\text{ for all }i\in I_\varepsilon
\bigr\}.
\]
At \(\varepsilon=0\) one recovers \(I_0\) and the exact spaces of the first
paragraph of the example. The associated kernel matrix has one extra row
\(y_i(\phi(x_i)^\top,1)\) for each index that enters \(I_\varepsilon\setminus I_0\).

\section{\texorpdfstring{Partial smoothness and the \(\mathcal{U}\)-Hessian}{Partial smoothness and the U-Hessian}}
\label{sec:ps}

The \(\mathcal{U}\)-space of Section~\ref{sec:prelim} is identified
below with the tangent space of an active manifold.
Subsection~\ref{ss:PS} records the manifold and the derivatives of
\(f\) along it.
Subsection~\ref{ss:UH-ident} identifies \(H_U\) and \(H_\varepsilon\)
with the Gram matrix of the covariant Hessian at \(\bar{x}\).
Subsection~\ref{ss:UH-along} writes the same matrix along \(\mathcal{M}\)
and records its continuity.
Subsection~\ref{ss:UH-tilt} tests positive-definiteness of the matrix at
\(\bar{x}\).
The same test is computed on the soft-margin example of
Section~\ref{sec:calculus}.

\subsection{Partial smoothness and derivatives along the manifold}
\label{ss:PS}

This subsection records the active manifold and the derivatives of
\(f\) along it. The next subsection uses those derivatives to identify
\(H_U\) and \(H_\varepsilon\) at \(\bar{x}\).

A set \(\mathcal{M}\subset\mathbb{R}^n\) is a \(C^k\)-smooth
manifold of codimension \(m\) around \(\bar{x}\in\mathcal{M}\) if there is
an open set \(Q\subset\mathbb{R}^n\) such that
\[
\mathcal{M}\cap Q=\bigl\{x\in Q:\phi_i(x)=0,\ i=1,\ldots,m\bigr\},
\]
where each \(\phi_i\) is of class \(C^k\) and
\(\{\nabla\phi_i(\bar{x}):i=1,\ldots,m\}\) is linearly independent. Then
\[
N_{\mathcal{M}}(\bar{x})
=
\spn\bigl\{\nabla\phi_i(\bar{x}):i=1,\ldots,m\bigr\},
\qquad
T_{\mathcal{M}}(\bar{x})=N_{\mathcal{M}}(\bar{x})^\perp.
\]

\begin{definition}[partial smoothness]
\label{def:PS}
Let \(f:\mathbb{R}^n\to\overline{\mathbb{R}}\) be proper and lsc, let
\(k\in\mathbb{N}\) with \(k\ge 1\), and let \(\mathcal{M}\) be a
\(C^k\)-smooth manifold around \(\bar{x}\in\mathcal{M}\). Following
Lewis \cite{Lewis2002active} and Lewis--Zhang
\cite[Definition~3.2]{lewis2013partial}, we say that \(f\) is
\(C^k\)-\emph{partly smooth} at \(\bar{x}\) relative to \(\mathcal{M}\)
if the following four properties hold.
\begin{enumerate}
\item[(i)]
\emph{Restricted smoothness.}
The restriction \(f|_{\mathcal{M}}\) is of class \(C^k\) near
\(\bar{x}\): there exists a representative \(h:\mathbb{R}^n\to\mathbb{R}\)
of class \(C^k\) near \(\bar{x}\) with \(h|_{\mathcal{M}}=f|_{\mathcal{M}}\)
locally around \(\bar{x}\).
\item[(ii)]
\emph{Regularity.}
At every point of \(\mathcal{M}\) near \(\bar{x}\), the function \(f\) is
subdifferentially regular in the sense of
\cite{Rockafellar1998} and \(\partial f\neq\emptyset\).
\item[(iii)]
\emph{Normal sharpness.}
\(N_{\mathcal{M}}(\bar{x})=\mathcal{V}f(\bar{x})\).
\item[(iv)]
\emph{Subgradient continuity.}
The mapping \(x\mapsto\partial f(x)\) is continuous at \(\bar{x}\) relative
to \(\mathcal{M}\): if \(x_k\in\mathcal{M}\) and \(x_k\to\bar{x}\), then
\(\partial f(x_k)\to\partial f(\bar{x})\) in the Painlev\'e--Kuratowski
sense.
\end{enumerate}
Condition (iii) and Definition~\ref{def:vu} give
\(\mathcal{U}f(\bar{x})=T_{\mathcal{M}}(\bar{x})\).
\end{definition}

\begin{definition}[prox-regularity and subdifferential continuity]
\label{def:prox-reg}
Let \(f:\mathbb{R}^n\to\overline{\mathbb{R}}\) be proper and let
\(\bar{v}\in\partial f(\bar{x})\). Following
\cite{Poliquin1996prox}, we say that \(f\) is \emph{prox-regular} at
\(\bar{x}\) for \(\bar{v}\) if \(f\) is locally lsc at \(\bar{x}\) and there
exist \(\varepsilon>0\) and \(r>0\) such that the quadratic minorant
\[
f(x')
\ge
f(x)+\langle v,x'-x\rangle-\frac{r}{2}\|x'-x\|^2
\]
holds for every pair \((x,v)\) with \(v\in\partial f(x)\),
\(\|x-\bar{x}\|<\varepsilon\),
\(\bigl|f(x)-f(\bar{x})\bigr|<\varepsilon\),
\(\|v-\bar{v}\|<\varepsilon\), and every \(x'\) with
\(\|x'-\bar{x}\|<\varepsilon\).
The same constants \(\varepsilon,r\) serve for all such pairs: the
minorant is uniform in the subgradient \(v\in\partial f(x)\) near
\(\bar{v}\).
The function is \emph{subdifferentially continuous} at \(\bar{x}\) for
\(\bar{v}\) if \(x^k\to\bar{x}\) and \(v^k\to\bar{v}\) with
\(v^k\in\partial f(x^k)\) imply \(f(x^k)\to f(\bar{x})\).
Convex functions are prox-regular and subdifferentially continuous at
every point where \(\partial f\neq\emptyset\).
\end{definition}

\begin{definition}[covariant derivative and covariant Hessian]
\label{def:cov-Hess}
Let \(\mathcal{M}\subset\mathbb{R}^n\) be a manifold around \(\bar{x}\)
and let \(f:\mathcal{M}\to\mathbb{R}\) be defined near \(\bar{x}\).
Following Lewis--Zhang \cite[Definition~2.11]{lewis2013partial}, if
\(\mathcal{M}\) is \(C^1\) and \(f|_{\mathcal{M}}\) is of class \(C^1\)
near \(\bar{x}\), the \emph{covariant derivative}
\(\nabla_{\mathcal{M}}f(\bar{x})\in T_{\mathcal{M}}(\bar{x})\) is the unique
vector satisfying
\[
\bigl\langle\nabla_{\mathcal{M}}f(\bar{x}),u\bigr\rangle
=\frac{d}{dt}\Big|_{t=0}f\bigl(P_{\mathcal{M}}(\bar{x}+tu)\bigr),
\qquad
u\in T_{\mathcal{M}}(\bar{x}).
\]
If \(\mathcal{M}\) is \(C^2\) and \(f|_{\mathcal{M}}\) is of class \(C^2\)
near \(\bar{x}\), the nearest-point projection is of class \(C^2\) near
the origin on \(T_{\mathcal{M}}(\bar{x})\), and the \emph{covariant Hessian}
\(\nabla^2_{\mathcal{M}}f(\bar{x}):T_{\mathcal{M}}(\bar{x})\times T_{\mathcal{M}}(\bar{x})\to\mathbb{R}\)
is the unique self-adjoint bilinear form satisfying
\[
\bigl\langle\nabla^2_{\mathcal{M}}f(\bar{x})\,u,u\bigr\rangle
=\frac{d^2}{dt^2}\Big|_{t=0}f\bigl(P_{\mathcal{M}}(\bar{x}+tu)\bigr),
\qquad
u\in T_{\mathcal{M}}(\bar{x}).
\]
\end{definition}

\begin{remark}[geodesic and Riemannian formulae]
\label{rem:cov-Ugrad-riem}
The pairing in Definition~\ref{def:cov-Hess} agrees with the geodesic
and Riemannian-gradient formulae of
Miller--Malick \cite[Section~4]{Miller2005}.
\end{remark}

\begin{lemma}[covariant derivative and \(\mathcal{U}\)-gradient]
\label{lem:cov-Ugrad}
Assume \(f\) is \(C^1\)-partly smooth at \(\bar{x}\) relative to
\(\mathcal{M}\). Then
\[
\bar g_u
=
U_f^\top\nabla_{\mathcal{M}}f(\bar{x}),
\qquad
g_U(\bar{x})
=
\nabla_{\mathcal{M}}f(\bar{x}).
\]
\end{lemma}
\begin{proof}
By Definition~\ref{def:PS}(iii) (normal sharpness),
\(\mathcal{U}f(\bar{x})=T_{\mathcal{M}}(\bar{x})\).
By Definition~\ref{def:PS}(i) (restricted smoothness),
\(f'(\bar{x};\tau)=\langle\nabla_{\mathcal{M}}f(\bar{x}),\tau\rangle\)
for all \(\tau\in T_{\mathcal{M}}(\bar{x})\).
Definition~\ref{def:PS}(ii) (regularity) gives
\(\partial f(\bar{x})\neq\emptyset\) and identifies the directional
derivative with the support function of \(\partial f(\bar{x})\),
so \(f'(\bar{x};d)\ge\langle g,d\rangle\) for every
\(g\in\partial f(\bar{x})\) and every direction \(d\).
Taking \(d=\pm\tau\) forces
\(\langle g-\nabla_{\mathcal{M}}f(\bar{x}),\tau\rangle=0\) on
\(T_{\mathcal{M}}(\bar{x})\), hence
\(P_{T_{\mathcal{M}}(\bar{x})}g=\nabla_{\mathcal{M}}f(\bar{x})\).
Definition~\ref{def:U-grad} gives
\(\bar g_u=U_f^\top g\), and \(U_f^\top P_{T_{\mathcal{M}}(\bar{x})}=U_f^\top\),
so \(\bar g_u=U_f^\top\nabla_{\mathcal{M}}f(\bar{x})\).
Definition~\ref{def:PS}(iv) is not used.
\end{proof}
When the composite models of Section~\ref{sec:calculus} are in addition
\(C^1\)-partly smooth, Lemma~\ref{lem:cov-Ugrad} reads the
\(\mathcal{U}\)-gradient formulae recorded after each rule
(the pairing \(\bar g_u(f)=U_f^\top F'(\bar{x})^\top s\) after
Theorem~\ref{thm:chain}, and the five analogues that follow)
as the coordinate form of the covariant chain rule
\(\nabla_{\mathcal{M}}f=P_{T_{\mathcal{M}}}F'(\bar{x})^\top\nabla_{\mathcal{M}_h}h\)
on the active manifolds of \(f\) and of \(h\). No new
first-order identity is involved.

\begin{definition}[Mordukhovich Hessian]
\label{def:MH}
Let \(f:\mathbb{R}^n\to\overline{\mathbb{R}}\) be proper and lsc, and let
\(\bar{v}\in\partial f(\bar{x})\).
The \emph{Mordukhovich Hessian} of \(f\) at \(\bar{x}\) for \(\bar{v}\)
is the set-valued mapping
\(\partial^2 f(\bar{x}\mid\bar{v}):\mathbb{R}^n\rightrightarrows\mathbb{R}^n\)
given by
\[
\partial^2 f(\bar{x}\mid\bar{v})(w)
=
\bigl\{
z\in\mathbb{R}^n
:
(z,-w)\in N_{\gph\partial f}(\bar{x},\bar{v})
\bigr\}.
\]
This is the coderivative of \(\partial f\) at \((\bar{x},\bar{v})\) in the
sense of \cite[Definition~8.33]{Rockafellar1998}:
\(\partial^2 f(\bar{x}\mid\bar{v})=D^*(\partial f)(\bar{x}\mid\bar{v})\).
\end{definition}

\subsection{Identification of \texorpdfstring{\(H_U\)}{HU} and \texorpdfstring{\(H_\varepsilon\)}{Heps}}
\label{ss:UH-ident}

The previous subsection gives the derivatives of \(f\) along
\(\mathcal{M}\). This subsection writes \(\mathcal{M}\) as a graph
over \(\mathcal{U}f(\bar{x})\) and identifies \(H_U\) and
\(H_\varepsilon\) at that point.

We recall \cite[Theorem~6.1]{Lewis2002active}, excluding the
equivalence with strong criticality recorded at the end of that theorem.
A local minimizer \(\bar{x}\) of a function \(h\) is \emph{sharp} if
\[
\liminf_{z\to 0}\frac{h(\bar{x}+z)-h(\bar{x})}{\|z\|}>0.
\]

\begin{lemma}[{Lewis graph representation \cite[Theorem~6.1]{Lewis2002active}}]
\label{lem:Lewis-graph}
Suppose the finite-valued function \(f:\mathbb{R}^n\to\mathbb{R}\) is partly smooth at
the point \(\bar{x}\) relative to the set \(\mathcal{M}\subset\mathbb{R}^n\).
Define subspaces \(\mathcal{U}=T_{\mathcal{M}}(\bar{x})\) and
\(\mathcal{V}=N_{\mathcal{M}}(\bar{x})\). Then there exists a function
\(v:\mathcal{U}\to\mathcal{V}\) with the following three properties.
\begin{enumerate}
\item[\rm (i)]
the function \(v\) is smooth near the origin;
\item[\rm (ii)]
for small vectors \(u\in\mathcal{U}\) and \(w\in\mathcal{V}\),
\(\bar{x}+u+w\in\mathcal{M}\) if and only if \(w=v(u)\);
\item[\rm (iii)]
\(v(u)=O(\|u\|^2)\) for small \(u\in\mathcal{U}\).
\end{enumerate}
Fix any vector \(y\in\ri\partial f(\bar{x})\). Then for any small vector
\(u\in\mathcal{U}\), the function
\begin{equation}
\label{eq:Lewis-slice}
w\in\mathcal{V}\mapsto f(\bar{x}+u+w)-\langle y,\bar{x}+u+w\rangle
\end{equation}
has a sharp minimizer at the point \(v(u)\).
\end{lemma}

\begin{lemma}[reduction along the graph]
\label{lem:phi-Hess}
Let \(f:\mathbb{R}^n\to\overline{\mathbb{R}}\) be proper and lsc, and
\(C^2\)-partly smooth at a point \(x\) relative to a \(C^2\)-smooth manifold
\(\mathcal{M}\).
Let \(U\) be an orthonormal frame of \(\mathcal{U}f(x)=T_{\mathcal{M}}(x)\)
and \(V\) a basis matrix of \(\mathcal{V}f(x)=N_{\mathcal{M}}(x)\).
Let \(x(u)=x+Uu+Vv(u)\) be the coordinate graph furnished by
Lemma~\ref{lem:Lewis-graph} at \(x\), and set \(\varphi(u):=f(x(u))\).
\begin{enumerate}
\item[\rm (i)]
\(\varphi\) is of class \(C^2\) near the origin.
\item[\rm (ii)]
If \(y\in\ri\partial f(x)\), then for every small \(u\) the slice
\(w\mapsto f(x+Uu+Vw)-\langle y,x+Uu+Vw\rangle\) has a sharp
minimizer at \(w=v(u)\).
If \(y=g_U(x)\), the pairing vanishes on \(\mathcal{V}f(x)\) and the
slice \(w\mapsto f(x+Uu+Vw)\) has a sharp minimizer at
\(w=v(u)\).
\item[\rm (iii)]
\[
\nabla^2\varphi(0)
=
U^\top\nabla^2_{\mathcal{M}}f(x)\,U.
\]
\end{enumerate}
In particular, if \(x=\bar{x}\) and \(0\in\ri\partial f(\bar{x})\), then
\(g_U(\bar{x})=0\), the frames may be taken to be \(U_f\) and \(V_f\),
and (ii) is the slice \(w\mapsto f(\bar{x}+U_fu+V_fw)\).
\end{lemma}
\begin{proof}
Definition~\ref{def:PS}(iii) and Definition~\ref{def:vu} give
\(\mathcal{V}f(x)=N_{\mathcal{M}}(x)\) and
\(\mathcal{U}f(x)=T_{\mathcal{M}}(x)\).
Lemma~\ref{lem:Lewis-graph} at \(x\) yields
\(v_{\mathcal{M}}:\mathcal{U}f(x)\to\mathcal{V}f(x)\)
of class \(C^2\) near the origin with \(v_{\mathcal{M}}(p)=O(\|p\|^2)\).
The coordinate map defined by \(Vv(u)=v_{\mathcal{M}}(Uu)\) then
satisfies \(v(0)=0\), \(Dv(0)=0\) and \(v(u)=O(\|u\|^2)\), and
\(x(u)=x+Uu+Vv(u)\) traces \(\mathcal{M}\) near \(x\).
Restricted smoothness of \(f|_{\mathcal{M}}\) and smoothness of \(v\)
give (i).
Lemma~\ref{lem:Lewis-graph} with the given \(y\in\ri\partial f(x)\)
gives the paired slice in (ii).
If \(y=g_U(x)\), Definition~\ref{def:U-grad} places \(y\) in
\(\mathcal{U}f(x)\), so the pairing vanishes on \(\mathcal{V}f(x)\).
If \(x=\bar{x}\) and \(0\in\ri\partial f(\bar{x})\), then
\(g_U(\bar{x})=0\) and the slice without the pairing is the case \(y=0\).

For (iii), fix \(u\in\mathbb{R}^{m}\) with \(m=\dim\mathcal{U}f(x)\)
and set \(\gamma(t):=x(tu)\).
Then \(\gamma(t)\in\mathcal{M}\), \(\gamma(0)=x\), and
\(Dv(0)=0\) yields \(\gamma'(0)=Uu\). Since \(v(tu)=O(t^2)\), \(\gamma(t)=x+tUu+O(t^2)\).
The nearest-point curve
\(\pi(t)=P_{\mathcal{M}}(x+tUu)\) has the same expansion, so
\(\gamma\) and \(\pi\) have the same 2-jet at \(t=0\).
By Definition~\ref{def:cov-Hess},
\[
\frac{d}{dt}\Big|_{t=0}f\bigl(\pi(t)\bigr)
=
\bigl\langle\nabla_{\mathcal{M}}f(x),\pi'(0)\bigr\rangle
=
\bigl\langle\nabla_{\mathcal{M}}f(x),Uu\bigr\rangle.
\]
Definition~\ref{def:cov-Hess} places
\(\nabla_{\mathcal{M}}f(x)\) in \(T_{\mathcal{M}}(x)\),
while \(\pi''(0)\) lies in \(N_{\mathcal{M}}(x)\) because
\(\pi\) is the projection of a straight line in the tangent space.
The pairing \(\langle\nabla_{\mathcal{M}}f(x),\pi''(0)\rangle\)
therefore vanishes by orthogonality.
Differentiating once more,
\[
\frac{d^2}{dt^2}\Big|_{t=0}f\bigl(\pi(t)\bigr)
=
\bigl\langle\nabla^2_{\mathcal{M}}f(x)\,Uu,\,Uu\bigr\rangle
+
\bigl\langle\nabla_{\mathcal{M}}f(x),\pi''(0)\bigr\rangle
=
\bigl\langle\nabla^2_{\mathcal{M}}f(x)\,Uu,\,Uu\bigr\rangle.
\]
The common 2-jet gives the same second derivative along \(\gamma\).
From \(\varphi(tu)=f(\gamma(t))\) we obtain
\[
\langle\nabla^2\varphi(0)\,u,u\rangle
=
\bigl\langle\nabla^2_{\mathcal{M}}f(x)\,Uu,\,Uu\bigr\rangle.
\]
Polarization of this identity of symmetric quadratic forms yields
\(\nabla^2\varphi(0)=U^\top\nabla^2_{\mathcal{M}}f(x)\,U\).
\end{proof}

\begin{proposition}[identification of the \(\mathcal{U}\)-Hessian]
\label{prop:UH-ident}
Let \(f:\mathbb{R}^n\to\overline{\mathbb{R}}\) be proper and lsc, and
\(C^2\)-partly smooth at \(\bar{x}\) relative to a \(C^2\)-smooth
manifold \(\mathcal{M}\). Assume \(0\in\ri\partial f(\bar{x})\).
Write
\[
H
\;:=\;
U_f^\top\nabla^2_{\mathcal{M}}f(\bar{x})\,U_f.
\]
\begin{enumerate}
\item[\rm (i)]
If \(f\in\Gamma_0(\mathbb{R}^n)\), then \(H_U=H\).
\item[\rm (ii)]
If \(f\) is prox-regular at \(\bar{x}\) for \(0\), then there exists
\(\varepsilon_0>0\) such that \(H_\varepsilon=H\) for every
\(\varepsilon\in(0,\varepsilon_0]\).
\end{enumerate}
In the convex case both identifications hold, and
\(H_U=H_\varepsilon\) for every such \(\varepsilon\).
\end{proposition}
\begin{proof}
Applied at \(\bar{x}\) with frames \(U_f\) and \(V_f\),
Lemma~\ref{lem:phi-Hess}(i) gives \(\varphi\in C^2\) near the origin,
Lemma~\ref{lem:phi-Hess}(iii) gives \(\nabla^2\varphi(0)=H\),
and Lemma~\ref{lem:phi-Hess}(ii) with \(0\in\ri\partial f(\bar{x})\)
says that \(w=v(u)\) is a sharp minimizer of the
slice \(w\mapsto f(\bar{x}+U_fu+V_fw)\).

For (i), the slice is convex in \(w\), so the sharp minimizer is the
unique global minimizer. Convexity and \(0\in\partial f(\bar{x})\) give
\(f\ge f(\bar{x})\), hence \(\mathcal{L}_U=\varphi\) near the origin
and \(H_U=\nabla^2\varphi(0)=H\).

For (ii), sharpness of the slice at \(u=0\) together with
\(v(u)=O(\|u\|^2)\) yields \(\varepsilon_0>0\) and a neighbourhood
\(\mathcal{O}\) of the origin such that, for every \(u\in\mathcal{O}\),
the point \(w=v(u)\) is the unique minimizer of the slice on the ball
\(\|w\|\le\varepsilon_0\).
Prox-regularity at \(\bar{x}\) for \(0\) supplies a quadratic minorant
\(f(x)\ge f(\bar{x})-\frac{\rho}{2}\|x-\bar{x}\|^2\) near \(\bar{x}\).
Shrinking \(\varepsilon_0\) if needed, for small \(u\) and
\(\|w\|\le\varepsilon_0\) the point
\(\bar{x}+U_fu+V_fw\) stays in that neighbourhood, so the slice is
bounded from below on the ball and \(\mathcal{L}_{\varepsilon_0}(u)\) is
finite. Combined with uniqueness of the minimizer,
\(\mathcal{L}_{\varepsilon_0}=\varphi\) on \(\mathcal{O}\).
The same holds with \(\varepsilon_0\) replaced by any
\(\varepsilon\in(0,\varepsilon_0]\), and
\(H_\varepsilon=\nabla^2\varphi(0)=H\).

If \(f\) is convex then it is prox-regular at \(\bar{x}\) for \(0\), so
both items apply and \(H_U=H_\varepsilon=H\).
\end{proof}

\subsection{Identification and continuity along \texorpdfstring{\(\mathcal{M}\)}{M}}
\label{ss:UH-along}

Subsection~\ref{ss:UH-ident} identifies \(H_U\) and \(H_\varepsilon\)
at the single point \(\bar{x}\). The same Gram matrix can be written at
nearby points of \(\mathcal{M}\) once the spaces, a frame, and the
reference subgradient move with the base point. The statements below
record those ingredients in the order of the items of
Definition~\ref{def:PS} they use, then identify the matrix along
\(\mathcal{M}\) and record its continuity.
For \(x\) near \(\bar{x}\) in \(\mathcal{M}\) write
\(\mathcal{U}_x:=T_{\mathcal{M}}(x)=\mathcal{U}f(x)\) and
\(\mathcal{V}_x:=N_{\mathcal{M}}(x)=\mathcal{V}f(x)\).
Let \(U(x)\) be an orthonormal frame of \(\mathcal{U}_x\) and
\(V(x)\) a basis matrix of \(\mathcal{V}_x\), and set
\(m=\dim\mathcal{U}f(\bar{x})\).

\begin{lemma}[spaces, frames, and the graph along \(\mathcal{M}\)]
\label{lem:along-setup}
Let \(\mathcal{M}\) be a \(C^1\)-smooth manifold around
\(\bar{x}\in\mathcal{M}\).
There exists a relative neighbourhood of \(\bar{x}\) in
\(\mathcal{M}\) on which \(P_{T_{\mathcal{M}}(x)}\) is continuous,
and of class \(C^1\) if \(\mathcal{M}\) is \(C^2\).

Now assume \(f:\mathbb{R}^n\to\overline{\mathbb{R}}\) is \(C^1\)-partly
smooth relative to \(\mathcal{M}\) throughout a neighbourhood of
\(\bar{x}\) in \(\mathcal{M}\). Then, on a relative neighbourhood:
\begin{enumerate}
\item[\rm (i)]
\(\mathcal{V}f(x)=N_{\mathcal{M}}(x)\),
\(\mathcal{U}f(x)=T_{\mathcal{M}}(x)\), and
\(\dim\mathcal{U}f(x)=m\) is constant.
If \(\mathcal{M}\) is \(C^2\), the frame \(U(x)\) may be chosen
continuously in \(x\).
\item[\rm (ii)]
\(g_U(x)=\nabla_{\mathcal{M}}f(x)\), and this map is continuous
along \(\mathcal{M}\).
\item[\rm (iii)]
If in addition \(f\) is \(C^2\)-partly smooth relative to a
\(C^2\)-smooth \(\mathcal{M}\) throughout that neighbourhood, then
Lemma~\ref{lem:phi-Hess}(i) and (iii) hold at each such \(x\), and
\(x\mapsto\nabla^2_{\mathcal{M}}f(x)\) is continuous along
\(\mathcal{M}\).
\end{enumerate}
\end{lemma}
\begin{proof}
Let \(\phi\) be a local equation of \(\mathcal{M}\) near \(\bar{x}\).
Then \(D\phi(x)\) has full row rank, and
\[
P_{T_{\mathcal{M}}(x)}
=
I-D\phi(x)^\top\bigl(D\phi(x)D\phi(x)^\top\bigr)^{-1}D\phi(x).
\]
If \(\phi\) is \(C^1\), the right-hand side is continuous; if \(\phi\)
is \(C^2\), it is of class \(C^1\).

Now assume \(C^1\)-partial smoothness throughout a relative
neighbourhood. Definition~\ref{def:PS}(iii) at each nearby \(x\)
gives \(\mathcal{V}f(x)=N_{\mathcal{M}}(x)\) and
\(\mathcal{U}f(x)=T_{\mathcal{M}}(x)\), so
\(\dim\mathcal{U}f(x)=\dim\mathcal{M}=m\).
If \(\mathcal{M}\) is \(C^2\), Gram--Schmidt on a continuous basis of
\(\ker D\phi(x)\) yields a continuous frame of
\(T_{\mathcal{M}}(x)\), hence of \(\mathcal{U}f(x)\). This is (i).

Lemma~\ref{lem:cov-Ugrad} at each such \(x\) uses
Definition~\ref{def:PS}(i)--(iii) and gives
\(g_U(x)=\nabla_{\mathcal{M}}f(x)\).
Definition~\ref{def:PS}(i) with \(k=1\) makes
\(\nabla_{\mathcal{M}}f\) continuous along \(\mathcal{M}\). This is (ii).

For (iii), Lemma~\ref{lem:phi-Hess}(i) and (iii) apply at each nearby
\(x\). The projection \(P_{\mathcal{M}}\) is \(C^2\) on a tubular
neighbourhood, so \(L=f\circ P_{\mathcal{M}}\) is \(C^2\) along
\(\mathcal{M}\) and \(x\mapsto\nabla^2L(x)\) is continuous.
Definition~\ref{def:cov-Hess} realises \(\nabla^2_{\mathcal{M}}f(x)\)
as the restriction of \(\nabla^2L(x)\) to \(T_{\mathcal{M}}(x)\).
\end{proof}
In particular, items (i) and (ii) give that
\(x\mapsto\bar g_u(x)=U(x)^\top\nabla_{\mathcal{M}}f(x)\)
is continuous along \(\mathcal{M}\).

\begin{lemma}[the parameter \(\bar g\) along \(\mathcal{M}\)]
\label{lem:gv-gU}
Assume \(f\) is \(C^1\)-partly smooth relative to \(\mathcal{M}\)
throughout a neighbourhood of \(\bar{x}\) in \(\mathcal{M}\), and
\(g_U(\bar{x})\in\ri\partial f(\bar{x})\). Let \(x\in\mathcal{M}\) be
sufficiently close to \(\bar{x}\), and let \(g_U(x)\) be the embedded
\(U\)-gradient of Definition~\ref{def:U-grad}.
\begin{enumerate}
\item[\rm (i)]
Then \(g_U(x)\in\ri\partial f(x)\) and the parameter
\(\bar g\) in Definitions~\ref{def:U-Lag-loc} and~\ref{def:U-Lag}, recentred at \(x\),
may be taken to be \(g_U(x)\).
\item[\rm (ii)]
With that choice,
\[
P_{\mathcal{V}f(x)}g_U(x)=0,
\qquad
\bar g_v(x)
=
V(x)^\dagger g_U(x)
=
0,
\]
so the pairings in Definitions~\ref{def:U-Lag} and~\ref{def:U-Lag-loc}
vanish:
\begin{align*}
\mathcal{L}_U^{g_U(x)}(u)
&=
\inf_w
f\bigl(x+U(x)u+V(x)w\bigr),
\\
\mathcal{L}_\varepsilon^{g_U(x)}(u)
&=
\inf_{\|w\|\le\varepsilon}
f\bigl(x+U(x)u+V(x)w\bigr).
\end{align*}
\end{enumerate}
If \(0\in\ri\partial f(\bar{x})\), then \(g_U(\bar{x})=0\) and the
hypothesis \(g_U(\bar{x})\in\ri\partial f(\bar{x})\) holds.
\end{lemma}
\begin{proof}
Definition~\ref{def:PS}(ii) gives regularity at every nearby point of
\(\mathcal{M}\), so \(\partial f(x)\) is nonempty, closed and convex
along \(\mathcal{M}\). Definition~\ref{def:PS}(iv) gives
continuity of \(\partial f|_{\mathcal{M}}\).
Lemma~\ref{lem:along-setup}(i)--(ii) give that
\(\dim\mathcal{V}f(x)\) is constant and that \(g_U\) is continuous
along \(\mathcal{M}\).
The argument that follows is that of
\cite[Theorem~2.12]{Miller2005}; global convexity of \(f\) on
\(\mathbb{R}^n\) is not used.

The first claim of Lemma~\ref{lem:along-setup} and item (i) yield a
continuous basis of \(\mathcal{V}f(x)\), hence linear bijections
\(\psi_x:\mathcal{V}f(x)\to\mathbb{R}^{n-m}\) depending continuously
on \(x\). Set
\(C(x):=\psi_x\bigl(\partial f(x)-g_U(x)\bigr)\). Continuity of
\(\partial f|_{\mathcal{M}}\), of \(g_U\) and of \(\psi_x\) yields
continuity of \(C\) along \(\mathcal{M}\). Since \(\psi_x\) is a
linear isomorphism,
\(g_U(x)\in\ri\partial f(x)\) if and only if \(0\in\inte C(x)\).

Suppose there is a sequence \(x_k\in\mathcal{M}\) with
\(x_k\to\bar{x}\) and \(g_U(x_k)\notin\ri\partial f(x_k)\). Then
\(0\notin\inte C(x_k)\). Each \(C(x_k)\) is convex, so there are
unit vectors \(s_k\in\mathbb{R}^{n-m}\) with
\(\langle s_k,y\rangle\le 0\) for all \(y\in C(x_k)\). Passing to a
subsequence, \(s_k\to s\) with \(\|s\|=1\). Since
\(0\in\inte C(\bar{x})\), some ball \(B(0,r)\) lies in
\(C(\bar{x})\). Continuity of \(C\) supplies \(v_k\in C(x_k)\) with
\(v_k\to v\) for every \(v\in B(0,r)\), whence
\(\langle s,v\rangle\le 0\). Thus \(s=0\), a contradiction.
Hence \(g_U(x)\in\ri\partial f(x)\) for \(x\in\mathcal{M}\) near
\(\bar{x}\).
Definition~\ref{def:U-Lag-loc} asks for a reference subgradient in
\(\ri\partial f(x)\), and Definition~\ref{def:U-Lag} asks for one in
\(\partial f(x)\), so \(\bar g=g_U(x)\) is admissible in both.
This is (i).

Definition~\ref{def:U-grad} puts \(g_U(x)\) in
\(\mathcal{U}f(x)=\mathcal{V}f(x)^\perp\), so the
\(\mathcal{V}\)-component displayed in (ii) is zero and
\(\bigl\langle\bar g_v(x),\,V(x)^\top V(x)w\bigr\rangle=0\).
\end{proof}
Proposition~\ref{prop:UH-ident} can now be repeated at nearby points of \(\mathcal{M}\).
By Lemma~\ref{lem:gv-gU} we write
\(\mathcal{L}_U^{x}:=\mathcal{L}_U^{g_U(x)}\) and
\(\mathcal{L}_\varepsilon^{x}:=\mathcal{L}_\varepsilon^{g_U(x)}\)
without the pairing. If the second derivatives exist, write
\(H_U(x):=\nabla^2\mathcal{L}_U^{x}(0)\)
and \(H_\varepsilon(x):=\nabla^2\mathcal{L}_\varepsilon^{x}(0)\). Write also
\[
H(x)
\;:=\;
U(x)^\top\nabla^2_{\mathcal{M}}f(x)\,U(x).
\]

\begin{proposition}[identification along \(\mathcal{M}\)]
\label{prop:UH-ident-along}
Assume \(f:\mathbb{R}^n\to\overline{\mathbb{R}}\) is \(C^2\)-partly
smooth relative to a \(C^2\)-smooth manifold \(\mathcal{M}\)
throughout a neighbourhood of \(\bar{x}\) in \(\mathcal{M}\).
Assume \(0\in\ri\partial f(\bar{x}\)).
There exists a relative neighbourhood \(\mathcal{N}\) of
\(\bar{x}\) in \(\mathcal{M}\) such that the following hold at every
\(x\in\mathcal{N}\).
\begin{enumerate}
\item[\rm (i)]
If \(f\in\Gamma_0(\mathbb{R}^n)\), then
\(H_U(x)=H(x)\) at every \(x\in\mathcal{N}\).
\item[\rm (ii)]
If \(f\) is prox-regular at a point \(x\in\mathcal{N}\) for
\(g_U(x)\), then there exists \(\varepsilon_0>0\) such that
\(H_\varepsilon(x)=H(x)\) for every \(\varepsilon\in(0,\varepsilon_0]\).
\end{enumerate}
\end{proposition}
\begin{proof}
The hypothesis \(0\in\ri\partial f(\bar{x})\) gives
\(g_U(\bar{x})=0\in\ri\partial f(\bar{x})\).
Lemma~\ref{lem:gv-gU} therefore yields a relative neighbourhood
\(\mathcal{N}\) of \(\bar{x}\) in \(\mathcal{M}\) on which
\(g_U(x)\in\ri\partial f(x)\).
Lemma~\ref{lem:gv-gU}(ii) drops the pairing, so
\(\mathcal{L}_U^{x}\) and \(\mathcal{L}_\varepsilon^{x}\) are defined
with reference subgradient \(g_U(x)\).
Lemma~\ref{lem:phi-Hess} at \(x\) with \(y=g_U(x)\) supplies
\(\varphi(u)=f\bigl(x+U(x)u+V(x)v(u)\bigr)\), a sharp minimizer of
the slice \(w\mapsto f\bigl(x+U(x)u+V(x)w\bigr)\) at \(w=v(u)\),
and \(\nabla^2\varphi(0)=H(x)\).

If \(f\in\Gamma_0(\mathbb{R}^n)\), the slice is convex in \(w\), so
the sharp minimizer is the unique global minimizer. Thus
\(\mathcal{L}_U^{x}=\varphi\) near the origin and
\(H_U(x)=H(x)\).

If \(f\) is prox-regular at \(x\) for \(g_U(x)\), the minorant is
\(f(z)\ge f(x)+\langle g_U(x),z-x\rangle-\frac{\rho_x}{2}\|z-x\|^2\)
near \(x\). The constrained slice is bounded from below on a ball
\(\|w\|\le\varepsilon_0\); shrinking the ball if needed,
sharpness and \(v(u)=O(\|u\|^2)\) make \(v(u)\) the unique
minimizer on that ball for small \(u\), hence
\(\mathcal{L}_\varepsilon^{x}=\varphi\) near the origin and
\(H_\varepsilon(x)=H(x)\) for every \(\varepsilon\in(0,\varepsilon_0]\).
\end{proof}

\begin{proposition}[continuity of the Gram and embedded \(\mathcal{U}\)-Hessians]
\label{prop:UH-Gram-cont}
\label{prop:UH-emb-cont}
Assume \(f:\mathbb{R}^n\to\overline{\mathbb{R}}\) is \(C^2\)-partly
smooth relative to a \(C^2\)-smooth manifold \(\mathcal{M}\)
throughout a neighbourhood of \(\bar{x}\) in \(\mathcal{M}\), and
\(0\in\ri\partial f(\bar{x}\)).
There exists a relative neighbourhood \(\mathcal{N}\) of
\(\bar{x}\) in \(\mathcal{M}\) on which the following hold.
Write
\[
B(x)
\;:=\;
P_{\mathcal{U}f(x)}\nabla^2_{\mathcal{M}}f(x)\,P_{\mathcal{U}f(x)}.
\]
Then \(x\mapsto B(x)\) is continuous on \(\mathcal{N}\) and does
not depend on the choice of frame.
\begin{enumerate}
\item[\rm (i)]
If \(f\in\Gamma_0(\mathbb{R}^n)\), then
\(x\mapsto H_U(x)\) is continuous on \(\mathcal{N}\) and
\(B(x)=U(x)H_U(x)U(x)^\top\).
\item[\rm (ii)]
If \(f\) is prox-regular at every \(x\in\mathcal{N}\) for
\(g_U(x)\), then
\(x\mapsto H_\varepsilon(x)\) is continuous on \(\mathcal{N}\) and
\(B(x)=U(x)H_\varepsilon(x)U(x)^\top\).
\end{enumerate}
\end{proposition}
\begin{proof}
Let \(\mathcal{N}\) be a relative neighbourhood on which
Proposition~\ref{prop:UH-ident-along} holds, shrunk if needed so that
Lemma~\ref{lem:along-setup}(i) supplies a continuous frame.
Lemma~\ref{lem:along-setup}(iii) gives that
\(x\mapsto\nabla^2_{\mathcal{M}}f(x)\) is continuous along
\(\mathcal{M}\), and the first claim of
Lemma~\ref{lem:along-setup} gives that
\(P_{T_{\mathcal{M}}(x)}\) is of class \(C^1\).
Thus \(x\mapsto H(x)=U(x)^\top\nabla^2_{\mathcal{M}}f(x)\,U(x)\)
and \(x\mapsto B(x)\) are continuous on \(\mathcal{N}\),
independently of convexity or prox-regularity.

View \(\nabla^2_{\mathcal{M}}f(x)\) as the self-adjoint operator on
\(\mathbb{R}^n\) that represents the covariant Hessian on
\(T_{\mathcal{M}}(x)\) and vanishes on \(N_{\mathcal{M}}(x)\).
If \(U(x)\) is any orthonormal frame of
\(\mathcal{U}f(x)=T_{\mathcal{M}}(x)\), then
\(U(x)U(x)^\top=P_{\mathcal{U}f(x)}\). By definition of \(H(x)\),
\[
U(x)\,H(x)\,U(x)^\top
=
B(x).
\]
The left-hand side does not depend on the frame: if
\(\widetilde U(x)=U(x)Q(x)\) with \(Q(x)\) orthogonal and
\(\widetilde H(x)=\widetilde U(x)^\top\nabla^2_{\mathcal{M}}f(x)\,\widetilde U(x)\),
then
\(\widetilde U\,\widetilde H\,\widetilde U^\top=U H U^\top\).
Proposition~\ref{prop:UH-ident-along}(i) gives \(H_U(x)=H(x)\) in
the convex case, and
Proposition~\ref{prop:UH-ident-along}(ii) gives
\(H_\varepsilon(x)=H(x)\) in the prox-regular case.
\end{proof}
The upshot is that the identification of
Subsection~\ref{ss:UH-ident} persists along \(\mathcal{M}\), and that
the resulting Gram matrix varies continuously in a continuous frame.

\subsection{Positive-definiteness, tilt stability, and growth}
\label{ss:UH-tilt}

Subsections~\ref{ss:UH-ident} and~\ref{ss:UH-along} identify \(H_U\)
and \(H_\varepsilon\) at \(\bar{x}\) and along \(\mathcal{M}\).
The test below uses only the matrix at \(\bar{x}\):
positive-definiteness is compared with tilt stability and with quadratic
growth.

\begin{definition}[tilt-stable local minimizer]
\label{def:tilt}
Following \cite{Poliquin1998tilt}, a point \(\bar{x}\) is a
\emph{tilt-stable local minimizer} of \(f\) if there exist \(\gamma>0\) and a
neighbourhood \(X\) of \(\bar{x}\) such that the mapping
\[
v\mapsto \argmin\bigl\{f(x)-\langle v,x\rangle:x\in X\bigr\}
\]
is single-valued and Lipschitz continuous on \(\gamma\mathbb{B}\) with value
\(\bar{x}\) at \(v=0\).
\end{definition}

\begin{definition}[strong metric regularity]
\label{def:smr}
A set-valued mapping \(S:\mathbb{R}^n\rightrightarrows\mathbb{R}^n\) is
\emph{strongly metrically regular} at \((\bar{x},\bar{y})\) if
\(\bar{y}\in S(\bar{x})\) and the inverse \(S^{-1}\) has a single-valued
Lipschitz localization around \((\bar{y},\bar{x})\). In that case
\(\lip(S^{-1})(\bar{y}\mid\bar{x})\) denotes the Lipschitz
modulus of this localization
\cite[Section~3G]{Dontchev2014}.
\end{definition}

\begin{definition}[quadratic growth and uniform quadratic growth]
\label{def:QG}
Let \(f:\mathbb{R}^n\to\overline{\mathbb{R}}\) be proper and lsc, and let
\(\bar{x}\) satisfy \(0\in\partial f(\bar{x})\). The function has
\emph{quadratic growth} at \(\bar{x}\) if there exist \(\kappa>0\) and a
neighbourhood \(X\) of \(\bar{x}\) such that
\[
f(x)\ge f(\bar{x})+\frac{\kappa}{2}\|x-\bar{x}\|^2
\qquad\text{for all }x\in X.
\]
Following \cite[Definition~1.1]{Drusvyatskiy2013tilt}, it has
\emph{uniform quadratic growth} at \(\bar{x}\) if there exist
\(\kappa>0\), \(\gamma>0\) and a neighbourhood \(X\) of \(\bar{x}\)
such that, for every \(v\in\gamma\mathbb{B}\), the tilt
\(f_v:=f-\langle v,\cdot\rangle\) admits a unique minimizer \(x_v\in X\),
with \(x_0=\bar{x}\), and
\begin{equation}
\label{eq:unif-QG}
f_v(x)
\ge
f_v(x_v)+\frac{\kappa}{2}\|x-x_v\|^2
\qquad\text{for all }x\in X.
\end{equation}
The quadratic lower bound is centred at the tilted minimizer \(x_v\),
not at \(\bar{x}\).
\end{definition}

The symbols \(\nabla^2_{\mathcal{M}}f(\bar{x})\) and
\(\partial^2 f(\bar{x}\mid 0)\) are those of
Definitions~\ref{def:cov-Hess} and~\ref{def:MH}.

\begin{lemma}[{Poliquin--Rockafellar tilt criterion \cite[Theorem~1.3]{Poliquin1998tilt}}]
\label{lem:PR98}
Let \(f:\mathbb{R}^n\to\overline{\mathbb{R}}\) be proper and lsc, and
finite at \(\bar{x}\). Assume that \(f\) is prox-regular and
subdifferentially continuous at
\(\bar{x}\) for \(0\in\partial f(\bar{x})\). Then \(\bar{x}\) is a
tilt-stable local minimizer of \(f\) if and only if
\begin{equation}
\label{eq:PR98-pd}
\langle z,w\rangle>0
\qquad\text{for all }w\neq 0\text{ and all }z\in\partial^2 f(\bar{x}\mid 0)(w).
\end{equation}
\end{lemma}

\begin{lemma}[{Lewis--Zhang generalized Hessian \cite[Theorem~5.3]{lewis2013partial}}]
\label{lem:LZ13-53}
Let \(f:\mathbb{R}^n\to\overline{\mathbb{R}}\) be proper and lsc, and let
\(\bar{x}\) satisfy \(0\in\ri\partial f(\bar{x})\). Assume that \(f\) is
\(C^2\)-partly smooth at \(\bar{x}\) relative to a \(C^2\)-smooth manifold
\(\mathcal{M}\), and that \(f\) is prox-regular and subdifferentially
continuous at \(\bar{x}\) for \(0\). Then, writing
\(\mathcal{V}f(\bar{x})=N_{\mathcal{M}}(\bar{x})\) and
\(\mathcal{U}f(\bar{x})=T_{\mathcal{M}}(\bar{x})\) by
Definition~\ref{def:PS}(iii),
\begin{equation}
\label{eq:LZ13}
\partial^2 f(\bar{x}\mid 0)(w)
=
\begin{cases}
\nabla^2_{\mathcal{M}}f(\bar{x})\,w+\mathcal{V}f(\bar{x}),
& w\in\mathcal{U}f(\bar{x}),\\
\emptyset,
& w\notin\mathcal{U}f(\bar{x}).
\end{cases}
\end{equation}
In particular \(\dom\partial^2 f(\bar{x}\mid 0)=\mathcal{U}f(\bar{x})\)
and
\[
\langle z,w\rangle
=\bigl\langle\nabla^2_{\mathcal{M}}f(\bar{x})\,w,w\bigr\rangle
\]
for every \(w\in\mathcal{U}f(\bar{x})\) and every
\(z\in\partial^2 f(\bar{x}\mid 0)(w)\).
\end{lemma}

\begin{theorem}[tilt stability versus positive-definiteness of the $\mathcal{U}$-Hessian]
\label{thm:tilt-UH}
Let \(f:\mathbb{R}^n\to\overline{\mathbb{R}}\) be proper and lsc, and let
\(\bar{x}\) satisfy \(0\in\ri\partial f(\bar{x})\).
Assume that \(f\) is prox-regular and subdifferentially continuous at
\(\bar{x}\) for \(0\), and \(C^2\)-partly smooth at \(\bar{x}\) relative
to a \(C^2\)-smooth manifold \(\mathcal{M}\).
The following are equivalent:
\begin{enumerate}
\item[\rm (a)] \(\bar{x}\) is a tilt-stable local minimizer of \(f\);
\item[\rm (b)]
\(\langle\nabla^2_{\mathcal{M}}f(\bar{x})w,w\rangle>0\)
for every \(w\in\mathcal{U}f(\bar{x})\setminus\{0\}\);
\item[\rm (c)] \(H_\varepsilon\succ 0\).
\end{enumerate}
If \(f\) is convex then \(H_U=H_\varepsilon\) by
Proposition~\ref{prop:UH-ident}(i), so (c) may be written \(H_U\succ 0\).
\end{theorem}
\begin{proof}
Definition~\ref{def:PS}(iii) gives \(\mathcal{V}f(\bar{x})=N_{\mathcal{M}}(\bar{x})\)
and \(\mathcal{U}f(\bar{x})=T_{\mathcal{M}}(\bar{x})\).
Lemma~\ref{lem:LZ13-53} yields
\(\dom\partial^2 f(\bar{x}\mid 0)=\mathcal{U}f(\bar{x})\) and
\[
\langle z,w\rangle
=
\langle\nabla^2_{\mathcal{M}}f(\bar{x})w,w\rangle
\]
for every \(w\in\mathcal{U}f(\bar{x})\) and every
\(z\in\partial^2 f(\bar{x}\mid 0)(w)\).
Thus \(\partial^2 f(\bar{x}\mid 0)\) is positive definite in the sense of
\eqref{eq:PR98-pd} if and only if (b) holds.
Lemma~\ref{lem:PR98} converts this into the equivalence of (a) and (b).
Proposition~\ref{prop:UH-ident} gives
\(H_\varepsilon=U_f^\top\nabla^2_{\mathcal{M}}f(\bar{x})\,U_f\), hence
\(\langle H_\varepsilon u,u\rangle=\langle\nabla^2_{\mathcal{M}}f(\bar{x})\,U_fu,U_fu\rangle\)
for all \(u\in\mathbb{R}^m\). Since the columns of \(U_f\) span
\(\mathcal{U}f(\bar{x})\), (b) and (c) are equivalent.
\end{proof}

\begin{remark}[sufficient \(\mathcal{U}\)-tests in~\cite{Liu2020}]
\label{rem:Liu2020-Thm8}
Theorem~8 of~\cite{Liu2020} is a one-sided test.
It assumes that a prox-regular function admits a fast track, that
\(\bar x\) is already a strict local minimizer of the tilt
\(h=f-\langle\bar g,\cdot\rangle\) for some
\(\bar g\in\ri\partial f(\bar x)\), and that the restriction of \(h\) to
the affine \(\mathcal{U}\)-space has positive-definite Hessian at the
origin. The conclusion is that \(\bar x\) is tilt-stable for \(h\).
The argument uses an implicit-function theorem on the fast track.
The matrix in that theorem is the ordinary Hessian of the affine
restriction \(u\mapsto h(\bar x+U_fu)\) at \(u=0\), written
\(\nabla_U^2 h(\bar x)\) in~\cite{Liu2020}.
That matrix is not identified there with \(H_\varepsilon\) or with the
Gram matrix \(H=U_f^\top\nabla^2_{\mathcal{M}}f(\bar x)\,U_f\).
Under the hypotheses of Theorem~\ref{thm:tilt-UH} one may take
\(\bar g=0\), so \(h=f\). Proposition~\ref{prop:UH-ident}(ii) then yields
\[
\nabla_U^2 f(\bar x)
=
H_\varepsilon
=
U_f^\top\nabla^2_{\mathcal{M}}f(\bar x)\,U_f.
\]
Positive-definiteness of any one of these three matrices is therefore
equivalent to that of the others. The one-sided implication of
\cite[Theorem~8]{Liu2020} becomes the equivalence of
Theorem~\ref{thm:tilt-UH}.
Example~\ref{ex:max-lq} shows that dropping \(0\in\ri\partial f(\bar{x})\)
breaks the equivalence in both directions of interest:
\(H_U\succ 0\) need not yield tilt stability, and a fast track may carry a
negative covariant Hessian.
\end{remark}

\begin{proposition}[strong metric regularity of \(\partial f\)]
\label{prop:metric-reg}
Let \(f:\mathbb{R}^n\to\overline{\mathbb{R}}\) be proper and lsc, and let
\(\bar{x}\) satisfy \(0\in\ri\partial f(\bar{x})\).
Assume that \(f\) is prox-regular and subdifferentially continuous at
\(\bar{x}\) for \(0\), and \(C^2\)-partly smooth at \(\bar{x}\) relative
to a \(C^2\)-smooth manifold \(\mathcal{M}\).
Then \(\partial f\) is strongly metrically regular at \((\bar{x},0)\) if and only if \(H_\varepsilon\succ 0\).
In that case
\[
\lip\bigl((\partial f)^{-1}\bigr)(0\mid\bar{x})=\|H_\varepsilon^{-1}\|.
\]
\end{proposition}
\begin{proof}
Theorem~\ref{thm:tilt-UH} gives \(H_\varepsilon\succ 0\) if and only if \(\bar{x}\) is
tilt-stable. Strong metric regularity of \(\partial f\) at
\((\bar{x},0)\) yields tilt stability by
\cite[Theorem~3.1]{Drusvyatskiy2013tilt}.
Conversely, tilt stability already supplies a local minimizer, so under
prox-regularity and subdifferential continuity at \(\bar{x}\) for \(0\),
\cite[Theorem~3.3]{Drusvyatskiy2013tilt} yields strong metric regularity.
For the modulus, write \(S:=(\partial f)^{-1}\). Strong metric regularity
of \(\partial f\) at \((\bar{x},0)\) means that \(S\) admits a
single-valued Lipschitz localisation around \((0,\bar{x})\); the quantity
\(\lip(S)(0\mid\bar{x})\) in the statement is the Lipschitz modulus of
that localisation, in the notation of
\cite[Theorem~9.40]{Rockafellar1998}.
Prox-regularity and subdifferential continuity of \(f\) at \(\bar{x}\)
for \(0\) make \(\gph\partial f\), hence also \(\gph S\), locally closed
at the reference point, so that theorem applies:
\[
\lip(S)(0\mid\bar{x})
=
\bigl|D^*S(0\mid\bar{x})\bigr|^{+},
\]
where \(|\,\cdot\,|^{+}\) is the outer norm of
\cite[Section~9D]{Rockafellar1998}.

By Definition~\ref{def:MH},
\(\partial^2 f(\bar{x}\mid 0)=D^*(\partial f)(\bar{x}\mid 0)\).
Lemma~\ref{lem:LZ13-53} and Proposition~\ref{prop:UH-ident} give
\[
\partial^2 f(\bar{x}\mid 0)(U_fu)
=
U_f H_\varepsilon u+\mathcal{V}f(\bar{x}),
\qquad
u\in\mathbb{R}^{m},
\]
while \(\partial^2 f(\bar{x}\mid 0)(w)=\emptyset\) for
\(w\notin\mathcal{U}f(\bar{x})\).
Thus \(z\in\partial^2 f(\bar{x}\mid 0)(w)\) if and only if
\(w=U_fu\) and \(U_f^\top z=H_\varepsilon u\). Since
\(H_\varepsilon\succ 0\), this solves uniquely as
\(w=U_f H_\varepsilon^{-1}U_f^\top z\). In other words, the inverse of
the set-valued map \(\partial^2 f(\bar{x}\mid 0)\) is the single-valued
linear map
\[
L
:
z
\mapsto
U_f H_\varepsilon^{-1}U_f^\top z.
\]
The graph of \(S\) is the graph of \(\partial f\) with the two factors
swapped, so
\((\xi,-z)\in N_{\gph S}(0,\bar{x})\) if and only if
\((-z,\xi)\in N_{\gph\partial f}(\bar{x},0)\).
Definition~\ref{def:MH} converts the latter into
\(-z\in\partial^2 f(\bar{x}\mid 0)(-\xi)\).
Lemma~\ref{lem:LZ13-53} realises the graph of
\(\partial^2 f(\bar{x}\mid 0)\) as the linear subspace
\[
\bigl\{
(U_fu,\,U_f H_\varepsilon u+v)
:
u\in\mathbb{R}^{m},\,
v\in\mathcal{V}f(\bar{x})
\bigr\},
\]
so \(\partial^2 f(\bar{x}\mid 0)(-w)=-\partial^2 f(\bar{x}\mid 0)(w)\)
for every \(w\). Hence \(z\in\partial^2 f(\bar{x}\mid 0)(\xi)\).
Therefore \(\xi\in D^*S(0\mid\bar{x})(z)\) if and only if
\(\xi=Lz\): the coderivative \(D^*S(0\mid\bar{x})\) coincides with \(L\).
The outer norm of a linear map is its operator norm, and
\(U_f^\top U_f=I_m\) yields \(\|L\|=\|H_\varepsilon^{-1}\|\), since
\(L\) vanishes on \(\mathcal{V}f(\bar{x})\) and acts as
\(H_\varepsilon^{-1}\) on \(\mathcal{U}f(\bar{x})\).
Hence
\(\lip\bigl((\partial f)^{-1}\bigr)(0\mid\bar{x})=\|H_\varepsilon^{-1}\|\).
\end{proof}

\begin{lemma}[quadratic growth versus uniform quadratic growth]
\label{lem:QG-unif}
Let \(f:\mathbb{R}^n\to\overline{\mathbb{R}}\) be proper and lsc, and let
\(\bar{x}\) satisfy \(0\in\ri\partial f(\bar{x})\).
Assume that \(f\) is prox-regular and subdifferentially continuous at
\(\bar{x}\) for \(0\), and \(C^2\)-partly smooth at \(\bar{x}\) relative
to a \(C^2\)-smooth manifold \(\mathcal{M}\).
Then \(f\) has quadratic growth at \(\bar{x}\) if and only if it has
uniform quadratic growth at \(\bar{x}\), in the sense of
Definition~\ref{def:QG}.
\end{lemma}
\begin{proof}
Under \(C^2\)-partial smoothness and prox-regularity at \(\bar{x}\) for
\(0\in\ri\partial f(\bar{x})\), Lewis--Zhang
\cite[Theorem~6.3]{lewis2013partial} equate tilt stability of
\(\bar{x}\) with ordinary quadratic growth of \(f\) near \(\bar{x}\).
Under prox-regularity and subdifferential continuity at \(\bar{x}\) for
\(0\), and once \(\bar{x}\) is known to be a local minimizer,
\cite[Theorem~3.3]{Drusvyatskiy2013tilt} equate tilt stability with
uniform quadratic growth in the sense of \eqref{eq:unif-QG}
(their stable strong local minimizer; equivalently, a quadratic minorant
centred at every nearby point of \(\gph\partial f\)). Ordinary quadratic
growth yields the local-minimizer hypothesis, so the two growth
conditions are each equivalent to tilt stability, hence to each other.
\end{proof}

\begin{remark}[quadratic growth, with and without partial smoothness]
\label{rem:QG-LO}
If \(f\in\Gamma_0(\mathbb{R}^n)\), \(0\in\ri\partial f(\bar{x})\), and
\(\mathcal{L}_U\) is twice differentiable at the origin with
\(H_U\succ 0\), then
\(\mathcal{L}_U(u)\ge\mathcal{L}_U(0)+\frac{c}{2}\|u\|^2\) for small
\(u\in\mathcal{U}f(\bar{x})\) and some \(c>0\).
\cite[Theorem~1]{Lemarechal2001Growth} transfers that lower bound to
\(f\), so \(f\) has quadratic growth at \(\bar{x}\) in the sense of
Definition~\ref{def:QG}, with no partial-smoothness assumption.

If in addition \(f\) is prox-regular and subdifferentially continuous at
\(\bar{x}\) for \(0\), and \(C^2\)-partly smooth at \(\bar{x}\) relative
to a \(C^2\)-smooth manifold \(\mathcal{M}\), then
\(H_\varepsilon\succ 0\) is equivalent to tilt stability of \(\bar{x}\)
by Theorem~\ref{thm:tilt-UH}, hence to uniform quadratic growth of \(f\)
at \(\bar{x}\) in the sense of Definition~\ref{def:QG}
\cite[Theorem~3.3]{Drusvyatskiy2013tilt}. If \(f\) is convex then
\(H_U=H_\varepsilon\), so \(H_U\succ 0\) already yields the uniform
statement.
\end{remark}

The next three examples isolate the two hypotheses that make
\(H_U\succ 0\) equivalent to tilt stability, then return to the
soft-margin model of Example~\ref{ex:probe}.
Example~\ref{ex:UH-pd-split} keeps \(0\in\ri\partial f\) and shows that
positive-definiteness cannot be weakened to positive-semidefiniteness;
Example~\ref{ex:max-lq} keeps a positive-definite matrix and drops the
relative-interior condition; Example~\ref{ex:probe-UH} realises both
pictures on the same structured objective, according as \(I_0\) is empty
or not.

\begin{example}[a split quadratic versus a flat direction]
\label{ex:UH-pd-split}
Set \(f(x,y)=\frac12 x^2+|y|\). Then
\(\partial f(0,0)=\{0\}\times[-1,1]\), so \(0\in\ri\partial f(0,0)\),
\(\mathcal{V}f(0,0)=\spn\{e_2\}\) and
\(\mathcal{U}f(0,0)=\spn\{e_1\}\).
The function is convex and \(C^\infty\)-partly smooth at the origin
relative to \(\mathcal{M}=\mathbb{R}\times\{0\}\).
In the frame \(U_f=e_1\) one has \(\mathcal{L}_U(u)=\frac12 u^2\), so
\(H_U=1\succ 0\). Theorem~\ref{thm:tilt-UH} yields tilt stability of the
origin, and Proposition~\ref{prop:metric-reg} yields
\(\lip((\partial f)^{-1})(0\mid(0,0))=1\).

By contrast, let \(g(x,y)=\frac12 x^2\). Then
\(\partial g(0,0)=\{(0,0)\}\), \(\mathcal{U}g(0,0)=\mathbb{R}^2\) and
\(H_U=\diag(1,0)\), which is not positive definite. The origin is a
minimizer, yet the whole \(y\)-axis consists of minimizers, so tilt
stability fails. See the case \(I_0=\emptyset\) in
Example~\ref{ex:probe-UH} below.
\end{example}

\begin{example}[linear-quadratic max: the relative-interior gap]
\label{ex:max-lq}
The two functions
\[
f_+(v,u)
=
\max\Bigl\{v,\,\frac{a}{2}u^2\Bigr\},
\qquad
f_-(v,u)
=
\max\Bigl\{v,\,-\frac{a}{2}u^2\Bigr\},
\qquad
a>0,
\]
share the same first-order data at the origin.
In both cases
\[
\partial f_\pm(0,0)
=
[0,1]\times\{0\},
\qquad
\mathcal{V}f_\pm(0,0)
=
\spn\{e_v\},
\qquad
\mathcal{U}f_\pm(0,0)
=
\spn\{e_u\},
\]
so \(0\in\partial f_\pm(0,0)\) but \(0\notin\ri\partial f_\pm(0,0)\).
Each function is \(C^\infty\)-partly smooth at the origin relative to the
coincidence set of the two pieces:
\[
\mathcal{M}_+
=
\bigl\{\bigl(\tfrac{a}{2}u^2,\,u\bigr):u\in\mathbb{R}\bigr\},
\qquad
\mathcal{M}_-
=
\bigl\{\bigl(-\tfrac{a}{2}u^2,\,u\bigr):u\in\mathbb{R}\bigr\}.
\]
Normal sharpness holds: \(N_{\mathcal{M}_\pm}(0,0)=\spn\{e_v\}=\mathcal{V}f_\pm(0,0)\).
The graph coordinates over \(\mathcal{U}\) are
\(v_+(u)=\frac{a}{2}u^2\) and \(v_-(u)=-\frac{a}{2}u^2\), so
\(Dv_\pm(0)=0\) and \(D^2v_\pm(0)=\pm a\).

On \(\mathcal{M}_+\) one has \(f_+|_{\mathcal{M}_+}=\frac{a}{2}u^2\),
hence \(\nabla^2_{\mathcal{M}_+}f_+(0,0)=a\) in the frame \(U_f=e_u\).
Since \(f_+\) is convex and \(0\in\partial f_+(0,0)\),
Definition~\ref{def:U-Lag} gives
\[
\mathcal{L}_U(u)
=
\inf_w\max\bigl\{w,\,\tfrac{a}{2}u^2\bigr\}
=
\tfrac{a}{2}u^2,
\qquad
H_U=a.
\]
For a general \(\bar g=(\gamma,0)\in\ri\partial f_+(0,0)\) the same computation
gives
\[
\mathcal{L}_U^{\bar g}(u)
=
(1-\gamma)\tfrac{a}{2}u^2,
\qquad
H_U^{\bar g}
=
a(1-\gamma).
\]
Thus \(H_U^{\bar g}\) depends on the reference subgradient whenever
\(P_{\mathcal{V}}\bar g\neq 0\), and equals \(H_U\) only for
\(\gamma=0\), which lies outside the relative interior.
The origin is a minimizer of \(f_+\), but it is not isolated:
\(f_+(v,0)=0\) for every \(v\le 0\).
Tilt stability therefore fails, even though \(H_U=a\succ 0\):
the hypothesis \(0\in\ri\partial f(\bar{x})\) is not met, and the
\(\mathcal{V}\)-slice is flat in the direction \(-e_v\).

On \(\mathcal{M}_-\) one has \(f_-|_{\mathcal{M}_-}=-\frac{a}{2}u^2\),
so the covariant Hessian is \(-a\).
The origin is not a local minimizer.
The local \(\mathcal{U}\)-Lagrangian of Definition~\ref{def:U-Lag}
with \(\bar g=(\gamma,0)\in\ri\partial f_-(0,0)\)
is
\[
\mathcal{L}_\varepsilon^{\bar g}(u)
=
(\gamma-1)\tfrac{a}{2}u^2
\qquad\text{for small }u,
\qquad
H_\varepsilon^{\bar g}
=
a(\gamma-1)<0.
\]
\end{example}

\begin{example}[Example~\ref{ex:probe} continued]
\label{ex:probe-UH}
Return to the soft-margin objective under the standing hypotheses of
Example~\ref{ex:probe} with \(I_>=\emptyset\) and with multipliers for the
active margins in \((0,C)\). The function is convex, hence prox-regular and
subdifferentially continuous. The active set
\(\mathcal{M}=\{(w,b):F_i(w,b)=0\text{ for all }i\in I_0\}\) is an affine
manifold of class \(C^\infty\), and \(f\) is \(C^2\)-partly
smooth at \((\bar w,\bar b)\) relative to \(\mathcal{M}\).
On \(\mathcal{M}\) the hinge terms vanish identically, so
\(f|_{\mathcal{M}}(w,b)=\frac12\|w\|^2\).
Differentiating twice in a frame \(U_f\) of
\(\mathcal{U}f(\bar w,\bar b)=T_{\mathcal{M}}(\bar w,\bar b)\) yields
\[
H_U
=
U_f^\top
\begin{pmatrix}
I_n&0\\
0&0
\end{pmatrix}
U_f.
\]
Partition \(U_f=\begin{pmatrix}U_w\\ u_b^\top\end{pmatrix}\) with
\(U_w\in\mathbb{R}^{n\times n_U}\) and \(n_U=\dim\mathcal{U}f(\bar w,\bar b)\). Then \(H_U=U_w^\top U_w\), so
\(H_U\succ 0\) if and only if \(U_w\) has full column rank.
A vector of \(\mathcal{U}f(\bar w,\bar b)\) has vanishing \(w\)-component if
and only if it is a pure bias direction \((0,db)\), which lies in
\(\mathcal{U}f\) if and only if \(y_i\,db=0\) for every \(i\in I_0\).
If \(I_0\neq\emptyset\), then \(y_i=\pm 1\) forces \(db=0\), so \(U_w\) is
injective and \(H_U\succ 0\).
Theorem~\ref{thm:tilt-UH} and Proposition~\ref{prop:metric-reg} yield
tilt stability and
\(\lip\bigl((\partial f)^{-1}\bigr)\bigl(0\bigm|(\bar w,\bar b)\bigr)
=\|(U_w^\top U_w)^{-1}\|\).
Near the origin, \(\mathcal{L}_U(u)=\frac12\|U_wu\|^2\).

If instead \(I_0=\emptyset\), then \(\mathcal{V}f=\{0\}\),
\(\mathcal{U}f=\mathbb{R}^{n+1}\) and \(H_U=\diag(I_n,0)\),
which is not positive definite. This is the same flat direction as for
\(g(x,y)=\frac12 x^2\) in Example~\ref{ex:UH-pd-split}.
\end{example}

\section{Concluding Remarks}
\label{sec:conclude}
The paper records four interfaces, so that a frame and a reduced Hessian
can be read from the factors of a composite model and then used as the
second-order test already present in tilt stability and in proximal-gradient
\(\mathcal{VU}\) methods.
A \(\mathcal{VU}\)-calculus for the six constructions of
Section~\ref{sec:calculus} writes the spaces, a frame of
\(\mathcal{U}f\), and the \(\mathcal{U}\)-gradient in terms of the
factors on \(\mathcal{R}_{h,F}\); the same linear algebra pushes a proper
outer approximation of \(\partial h\) forward to an
\(\varepsilon\)-\(\mathcal{VU}\) decomposition of \(f\).
If \(f\) is \(C^2\)-partly smooth at \(\bar{x}\) and
\(0\in\ri\partial f(\bar{x})\), then \(H_U\) (convex case) or
\(H_\varepsilon\) (prox-regular case) equals the Gram matrix
\(U_f^\top\nabla^2_{\mathcal{M}}f(\bar{x})\,U_f\).
That matrix, recentred along \(\mathcal{M}\), depends continuously on the
base point in a continuous frame.
Under \(C^2\)-partial smoothness, prox-regularity and subdifferential
continuity at \(\bar{x}\) for \(0\), and \(0\in\ri\partial f(\bar{x})\),
\(H_\varepsilon\succ 0\) is equivalent to tilt stability of \(\bar{x}\)
and to strong metric regularity of \(\partial f\) at \((\bar{x},0)\), with
\(\lip\bigl((\partial f)^{-1}\bigr)(0\mid\bar{x})=\|H_\varepsilon^{-1}\|\).
Without partial smoothness, \(H_U\succ 0\) already yields quadratic
growth of a convex \(f\); with partial smoothness the same test yields
uniform quadratic growth, in the sense of Definition~\ref{def:QG}.
Example~\ref{ex:UH-pd-split} isolates a positive-definite
\(\mathcal{U}\)-Hessian from a flat direction;
Example~\ref{ex:max-lq} shows that the matrix test is not equivalent to
tilt stability once \(0\) leaves the relative interior of
\(\partial f(\bar{x})\);
Example~\ref{ex:probe-UH} realises both pictures on the hinge composite,
according as \(I_0\) is empty or not.

Two extensions of the chain rule are left aside.
If the outer function is not convex, the inclusion in
\eqref{eq:CR} need not be an equality.
If the inner mapping is merely locally Lipschitz, \(F'(\bar{x})^\top\) is
replaced by the coderivative \(D^*F(\bar{x})\).

Continuity of the \(\varepsilon\)-\(\mathcal{VU}\) spaces in
\((x,\varepsilon)\) is not treated here; it is the content of
\cite{Liu2019Subdifferential} in the convex polyhedral setting.
The proximal-gradient \(\mathcal{VU}\) method of~\cite{Liu2025} is analysed
with a bivariate \(\mathcal{U}\)-Lagrangian \(\mathcal{L}_U(u,g_v)\) and a
partial \(\mathcal{U}\)-Hessian in the sense of that paper.
By \cite[Lemma~4.2]{Liu2025}, the choice \(\bar{g}_v=0\) recovers the
single-variable matrix \(H_U\) of Definition~\ref{def:U-Lag}, so the
hypothesis that this partial matrix be positive definite at
\((\bar{x},0)\) in the superlinear-rate theorem of~\cite{Liu2025} is,
under the standing second-order package of
Theorem~\ref{thm:tilt-UH}, equivalent to tilt stability of \(\bar{x}\).
In the notation of \cite[Theorem~4.4]{Liu2025}, superlinear convergence
further requires a Dennis--Mor\'e condition on the compressed inverse
Hessian, which holds as soon as the computed frames satisfy
\(U_k\to U\).
An \(\varepsilon\)-\(\mathcal{VU}\) frame with
\((x_k,\varepsilon_k)\to(\bar{x},0)\) supplies such a sequence precisely
when the spaces vary continuously.
Establishing that continuity for the composite calculus of
Section~\ref{sec:calculus}, beyond the enlargements of
\cite{Liu2019Subdifferential}, is left to a subsequent paper.

\section*{Statements and Declarations}

\begin{itemize}
\item \textbf{Funding.}
No funding was received for conducting this study.
\item \textbf{Competing interests.}
The author has no relevant financial or non-financial interests to disclose.
\item \textbf{Data availability.}
No datasets were generated or analysed in this study.
\end{itemize}

\bibliography{VUcalculus}

@Article{Mifflin2000pdg,
  author  = {Mifflin, Robert and Sagastiz{\'a}bal, Claudia},
  title   = {On {$\mathcal{VU}$}-theory for functions with primal-dual gradient structure},
  journal = {SIAM J. Optim.},
  volume  = {11},
  number  = {2},
  pages   = {547--571},
  year    = {2000},
  doi     = {10.1137/S1052623499350967}
}

@Article{Lemarechal2000,
  author  = {Lemar{\'e}chal, Claude and Oustry, Fran{\c{c}}ois and Sagastiz{\'a}bal, Claudia},
  title   = {The {$U$}-Lagrangian of a convex function},
  journal = {Trans. Amer. Math. Soc.},
  volume  = {352},
  number  = {2},
  pages   = {711--729},
  year    = {2000},
  doi     = {10.1090/S0002-9947-99-02243-6}
}

@Article{Lemarechal2001Growth,
  author  = {Lemar{\'e}chal, Claude and Oustry, Fran{\c{c}}ois},
  title   = {Growth conditions and {$U$}-Lagrangians},
  journal = {Set-Valued Anal.},
  volume  = {9},
  number  = {1-2},
  pages   = {123--129},
  year    = {2001}
}

@Article{Lewis2002active,
  author  = {Lewis, A. S.},
  title   = {Active sets, nonsmoothness, and sensitivity},
  journal = {SIAM J. Optim.},
  volume  = {13},
  number  = {3},
  pages   = {702--725},
  year    = {2002},
  doi     = {10.1137/S1052623401387623}
}

@Article{lewis2013partial,
  author  = {Lewis, Adrian S. and Zhang, Shida},
  title   = {Partial smoothness, tilt stability, and generalized Hessians},
  journal = {SIAM J. Optim.},
  volume  = {23},
  number  = {1},
  pages   = {74--94},
  year    = {2013},
  doi     = {10.1137/110852103}
}

@Article{Poliquin1996prox,
  author  = {Poliquin, R. A. and Rockafellar, R. T.},
  title   = {Prox-regular functions in variational analysis},
  journal = {Trans. Amer. Math. Soc.},
  volume  = {348},
  number  = {5},
  pages   = {1805--1838},
  year    = {1996}
}

@Article{Poliquin1998tilt,
  author  = {Poliquin, R. A. and Rockafellar, R. T.},
  title   = {Tilt stability of a local minimum},
  journal = {SIAM J. Optim.},
  volume  = {8},
  number  = {2},
  pages   = {287--306},
  year    = {1998},
  doi     = {10.1137/S1052623496309296}
}

@Article{Drusvyatskiy2013tilt,
  author  = {Drusvyatskiy, Dmitriy and Lewis, Adrian S.},
  title   = {Tilt stability, uniform quadratic growth, and strong metric regularity of the subdifferential},
  journal = {SIAM J. Optim.},
  volume  = {23},
  number  = {1},
  pages   = {256--267},
  year    = {2013},
  doi     = {10.1137/120876551}
}

@Article{Hare2001quadratic,
  author  = {Hare, Warren L. and Poliquin, R. A.},
  title   = {The quadratic sub-{L}agrangian of a prox-regular function},
  journal = {Nonlinear Anal.},
  volume  = {47},
  number  = {2},
  pages   = {1117--1128},
  year    = {2001},
  doi     = {10.1016/S0362-546X(01)00251-6}
}

@Article{Hare2004,
  author  = {Hare, Warren L. and Lewis, Adrian S.},
  title   = {Identifying active constraints via partial smoothness and prox-regularity},
  journal = {J. Convex Anal.},
  volume  = {11},
  number  = {2},
  pages   = {251--266},
  year    = {2004},
  url     = {https://www.heldermann-verlag.de/jca/jca11/jca0406.pdf}
}

@Article{Hare2020chain,
  author  = {Hare, Warren and Planiden, Chayne and Sagastiz{\'a}bal, Claudia},
  title   = {The chain rule for {$\mathcal{VU}$}-decompositions of nonsmooth functions},
  journal = {J. Convex Anal.},
  volume  = {27},
  number  = {1},
  pages   = {335--360},
  year    = {2020},
  eprint  = {1909.04799},
  archivePrefix = {arXiv},
  doi     = {10.48550/arXiv.1909.04799}
}

@Book{Rockafellar1998,
  author    = {Rockafellar, R. Tyrrell and Wets, Roger J.-B.},
  title     = {Variational Analysis},
  publisher = {Springer},
  address   = {Berlin},
  year      = {1998},
  series    = {Grundlehren der mathematischen Wissenschaften},
  volume    = {317}
}

@Article{Liu2020,
  author  = {Liu, Shuai and Eberhard, Andrew and Luo, Yousong},
  title   = {The {$\mathcal{U}$}-Lagrangian, fast track, and partial smoothness of a prox-regular function},
  journal = {Set-Valued Var. Anal.},
  volume  = {28},
  number  = {2},
  pages   = {369--394},
  year    = {2020},
  doi     = {10.1007/s11228-019-00518-z}
}

@Article{Liu2025,
  author  = {Liu, Shuai and Sagastiz{\'a}bal, Claudia and Solodov, Mikhail V.},
  title   = {Proximal gradient {$\mathcal{VU}$}-method with superlinear convergence for nonsmooth convex optimization},
  journal = {SIAM J. Optim.},
  volume  = {35},
  number  = {3},
  pages   = {1601--1629},
  year    = {2025},
  doi     = {10.1137/24M1697001}
}

@InCollection{Liu2019Subdifferential,
  author    = {Liu, Shuai and Sagastiz{\'a}bal, Claudia and Solodov, Mikhail},
  title     = {Subdifferential enlargements and continuity properties of the {$\mathcal{VU}$}-decomposition in convex optimization},
  booktitle = {Nonsmooth Optimization and Its Applications},
  editor    = {Hosseini, Seyedehsomayeh and Mordukhovich, Boris S. and Uschmajew, Andr{\'e}},
  series    = {International Series of Numerical Mathematics},
  volume    = {170},
  pages     = {55--87},
  publisher = {Springer},
  address   = {Cham},
  year      = {2019},
  doi       = {10.1007/978-3-030-11370-4_4}
}

@InCollection{Liu2020VUMethods,
  author    = {Liu, Shuai and Sagastiz{\'a}bal, Claudia},
  title     = {Beyond first order: {$\mathcal{VU}$}-decomposition methods},
  booktitle = {Numerical Nonsmooth Optimization},
  editor    = {Bagirov, Adil M. and Gaudioso, Manlio and Karmitsa, Napsu and M{\"a}kel{\"a}, Marko M. and Taheri, Sona},
  pages     = {297--329},
  publisher = {Springer},
  address   = {Cham},
  year      = {2020},
  doi       = {10.1007/978-3-030-34910-3_9}
}

@Article{Cortes1995,
  author  = {Cortes, Corinna and Vapnik, Vladimir},
  title   = {Support-vector networks},
  journal = {Mach. Learn.},
  volume  = {20},
  number  = {3},
  pages   = {273--297},
  year    = {1995},
  doi     = {10.1023/A:1022627411411}
}

@Article{Mifflin2003pdg,
  author  = {Mifflin, Robert and Sagastiz{\'a}bal, Claudia},
  title   = {Primal-dual gradient structured functions: second-order results; links to epi-derivatives and partly smooth functions},
  journal = {SIAM J. Optim.},
  volume  = {13},
  number  = {4},
  pages   = {1174--1194},
  year    = {2003},
  doi     = {10.1137/S1052623402412441}
}

@Article{Mifflin2005,
  author  = {Mifflin, Robert and Sagastiz{\'a}bal, Claudia},
  title   = {A {$\mathcal{VU}$}-algorithm for convex minimization},
  journal = {Math. Program.},
  volume  = {104},
  number  = {2-3},
  pages   = {583--608},
  year    = {2005},
  doi     = {10.1007/s10107-005-0630-3}
}

@Article{Miller2005,
  author  = {Miller, Stephen A. and Malick, J{\'e}r{\^o}me},
  title   = {Newton methods for nonsmooth convex minimization: connections among {$U$}-Lagrangian, {R}iemannian {Newton} and {SQP} methods},
  journal = {Math. Program.},
  volume  = {104},
  number  = {2-3},
  pages   = {609--633},
  year    = {2005},
  doi     = {10.1007/s10107-005-0626-z}
}

@Article{Shapiro2003,
  author  = {Shapiro, Alexander},
  title   = {On a class of nonsmooth composite functions},
  journal = {Math. Oper. Res.},
  volume  = {28},
  number  = {4},
  pages   = {677--692},
  year    = {2003},
  doi     = {10.1287/moor.28.4.677.20512}
}

@Book{Dontchev2014,
  author    = {Dontchev, Asen L. and Rockafellar, R. Tyrrell},
  title     = {Implicit Functions and Solution Mappings},
  subtitle  = {A View from Variational Analysis},
  edition   = {2},
  publisher = {Springer},
  address   = {New York},
  year      = {2014},
  series    = {Springer Series in Operations Research and Financial Engineering},
  doi       = {10.1007/978-1-4939-1037-3}
}
\end{document}